\documentclass
{amsart}

\usepackage{amsmath,amsthm,amsfonts,amssymb,mathtools}
\usepackage{mathrsfs}
\makeatletter
\@namedef{subjclassname@2020}{\textup{2020} Mathematics Subject Classification}
\makeatother
\usepackage
{hyperref}
\usepackage{capt-of}
\usepackage{relsize}

\newcommand{\Be}{\begin{equation}}
\newcommand{\Ee}{\end{equation}}
\newcommand{\Bea}{\begin{eqnarray}}
\newcommand{\Eea}{\end{eqnarray}}
\newcommand{\Bel}{\begin{align}}
\newcommand{\Eel}{\end{align}}
\newcommand{\Beas}{\begin{eqnarray*}}
\newcommand{\Eeas}{\end{eqnarray*}}
\newcommand{\Benu}{\begin{enumerate}}
\newcommand{\Eenu}{\end{enumerate}}
\newcommand{\Bi}{\begin{itemize}}
\newcommand{\Ei}{\end{itemize}}
\newcommand\supp{\operatorname{supp}}

\def\R{{\mathbb R}}

\theoremstyle{plain}
\newtheorem{thm}{Theorem}[section]
\newtheorem{cor}[thm]{Corollary}
\newtheorem{lem}[thm]{Lemma}
\newtheorem{prop}[thm]{Proposition}

\theoremstyle{remark}
\newtheorem{rmk}{Remark}  
\theoremstyle{definition}

\numberwithin{equation}{section}

\newcommand{\RNum}[1]{\uppercase\expandafter{\romannumeral #1\relax}}

\newcommand{\bxi}{{\bar\xi}}
\newcommand{\bx}{\bar x}

\newcommand{\vthe}{{\mathbf v_{\!\theta}}}
\newcommand{\vphi}{{\mathbf v_{\!\phi}}}
\newcommand{\cA}{\mathcal A}

\newcommand{\bA}{\mathbb A}

\newcommand{\bI}{\mathbb I}
\usepackage{wrapfig}
\usepackage{tikz}

\newcommand{\cL}{\mathcal L}
\newcommand{\sobn}[1]{\cL^{p}_{\alpha}(\mathbb D_{#1})}
\newcommand{\tchi}{{\tilde\chi}}
\newcommand{\cf}{{\mathcal F}}
\newcommand{\sspan}{\operatorname{span}}
\newcommand{\tx}{\textstyle}

\begin{document}

\title[Two-parameter  averages over tori]{$L^p$ maximal bound and Sobolev regularity of two-parameter  averages over tori}

\author[Juyoung Lee]{Juyoung Lee}
\author[Sanghyuk Lee]{Sanghyuk Lee}

\address{Department of Mathematical Sciences and RIM, Seoul National University, Seoul 08826, Republic of  Korea}
\email{ljy219@snu.ac.kr}
\email{shklee@snu.ac.kr}


\begin{abstract}   
We investigate  $L^p$  boundedness of the maximal function defined by  the  averaging operator   $f\to \cA_t^s f$  over  
  the two-parameter family of tori  $\mathbb{T}_t^{s}:=\{ ( (t+s\cos\theta)\cos\phi,\,(t+s\cos\theta)\sin\phi,\, s\sin\theta ): \theta, \phi \in [0,2\pi) \}$ with $c_0t>s>0$ for some $c_0\in (0,1)$.  We prove that the associated (two-parameter) maximal function  is bounded on $L^p$ if and only if $p>2$. We also obtain  $L^p$--$L^q$ estimates for the local maximal operator on a sharp range of $p,q$.  Furthermore, the  sharp smoothing estimates are proved including the sharp  local smoothing estimates for the operators $f\to \mathcal A_t^s f$  and   $f\to \mathcal A_t^{c_0t} f$. 
For the purpose, we make use of  Bourgain--Demeter's decoupling inequality for the cone and Guth--Wang--Zhang's local smoothing estimates for the $2$ dimensional wave operator. 
\end{abstract}

\maketitle


\section{Introduction}
The maximal functions generated by (one-parameter) dilations of a given  hypersurface have been extensively studied (for example, \cite[Ch. 11]{Stein}, \cite{nsw, is, IKM, gr, BDIM}, and references therein)  since Stein's seminal work on the spherical maximal function \cite{Stein2}. Most  of investigations were restricted to the one-parameter maximal functions. Meanwhile, the maximal operators involved with  more than one-parameter family of dilations  were considered by some authors (see \cite{RS} for multiparameter lacunary maximal functions and \cite{E, PS} for related results). For example,  Cho \cite{C} and Heo \cite{H}  obtained such results built on the $L^2$ method which requires sufficient  
decay of the Fourier transform of the associated surface measures.  
However, in those results, boundedness on  sharp range is generally unknown.  
Two-parameter maximal functions associated to homogeneous surfaces were studied by 
Marletta--Ricci \cite{MR}, and Marletta--Ricci--Zienkiewicz \cite{MRZ}, who obtained their boundedness on the sharp range. In those works, homogeneity makes it possible to deduce $L^p$ boundedness from that  of a one-parameter maximal operator.  Not much is so far known about the maximal functions which are genuinely of multiparameter.

In this paper we are concerned with a maximal function which is generated by  averages over  a natural tow-parameter family of tori in $\R^3$. 
Let us set 
\[ \Phi_t^s(\theta,\phi)=\big( (t+s\cos\theta)\cos\phi,\,(t+s\cos\theta)\sin\phi,\, s\sin\theta \big).\] 
For $0<s<t$,  we denote   $\mathbb{T}_t^{s}=\big\{ \Phi_t^{s}(\theta,\phi): \theta,\phi\in [0,2\pi) \big\},$
which is a parametrized torus in $\R^3$.   We consider a measure on $\mathbb{T}_t^{s}$ which is given by 
\Be\label{para} 
\langle   f,  \sigma_t^{s} \rangle =  \int_{[0,2\pi)^2}   f(  \Phi_t^{s}(\theta,\phi))\, d\theta d\phi. 
\Ee
Convolution with the measure $\sigma_t^{s}$ gives rise to a 2-parameter averaging operator $ \mathcal A_t^{s}f:=f\ast \sigma_t^{s}$. 
Let $0<c_0< 1$ be a fixed constant. We begin our discussion with the maximal operator  
\[ 
f\to \sup_{0<t } |\mathcal A_{t}^{c_0t}f|, 
\]
which is generated by the averages over (isotropic) dilations of the torus  $\mathbb T_{1}^{c_0}$. It is not difficult to see that
$f\to \sup_{0<t } |\mathcal A_{t}^{c_0t}f|$ is bounded on $L^p$ if and only if $p>2$. 
Indeed, writing  $ f\ast \sigma_t^{c_0t}=  \int  f\ast \mu_{t}^{\phi}\,  d\phi$, where  
$\mu_{t}^{\phi}$ is the measure on the circle 
$\{t\Phi_1^{c_0}(\phi,\theta):  \theta\in [0,2\pi)\}$. 
Since these circles are subsets of  2-planes containing the origin, $L^p$ boundedness of $ f\to  \sup_{t>0} |f\ast \mu_{t}^{\phi}|$ for $p>2$ can be obtained using the circular maximal theorem \cite{B2}.
 In fact, we need $L^p$ boundedness of the maximal function given  by the convolution averages in $\mathbb R^2$ over the circles $ \mathrm C((t/c_0) e_1, t)$,  
 which are not centered at the origin. Here, $\mathrm C(y, r)$ denotes the circle $\{x\in \mathbb R^2: |x-y|=r\}$. However, such a maximal estimate can be obtained by making use of  the local smoothing estimate for the wave operator (see, for example, \cite{MSS2}).  Failure of $L^p$ boundedness of $f\to \sup_{0<t } |\mathcal A_{t}^{c_0t}f|$ for $p\le 2$ follows if one  takes  $f(x)=\tilde \chi(x) |x_3|^{-1/2}|\log |x_3||^{-1/2-\epsilon}$ for a small $\epsilon>0$, where $\tilde \chi$ is a smooth positive function supported in 
a neighborhood of the origin.

 In the study of the averaging operator defined by hypersurface,  nonvanishing curvature of the underlying surface plays a crucial role.  However, the torus  $\mathbb T_{1}^{c_0}$ has vanishing curvature. More precisely,  
the Gaussian curvature $K(\theta,\phi)$  of $\mathbb{T}_1^{c_0}$ at the point $\Phi_1^{c_0}(\theta,\phi)$ is given by 
\[ K(\theta,\phi)=\frac{\cos\theta}{c_0(1+c_0\cos\theta)}. \]
Notice that $K$ vanishes on the circles  $\Phi_1^{c_0}(\pm \pi/2,\phi)$, $\phi\in[0, 2\pi)$. Decomposing $\mathbb T_{1}^{c_0}$  into the parts which are  away from and near those  circles, we can show, in an alternative way, 
$L^p$ boundedness of $ f\to \sup_{0<t } |\mathcal A_{t}^{c_0t}f|$ for $p>2$. The  part away from the circles  has nonvanishing curvature. Thus,  the associated maximal function is bounded on $L^p$ for $p>3/2$ (\cite{Stein2}). 
Meanwhile, the other parts near the circles  can be handled by the result in \cite{IKM}.

\subsection*{2-parameter maximal function} 
We now consider a two-parameter maximal function 
\[ \mathcal {M}f({x})=\sup_{0<s<c_0 t}\big| \mathcal A_{t}^sf({x}) \big|. \]
Here, the supremum is taken over  on the set $\{(t,s): 0<s<c_0 t \}$ so that  $\mathbb T_s^t$ remains to be a torus.  Unlike the one-parameter maximal function,  (nontrivial) $L^p$ on  $\mathcal {M}$ can not be obtained by the same argument as  above which relies $L^p$ boundedness of a 
related circular maximal function in $\mathbb R^2$.  In fact, to carry out the same argument, one needs $L^p$ boundedness of the maximal function 
given by the (convolution)  averages over the circles $\mathrm C(se_1, t)$ while supremum is taken over $0<s<c_0t$. However, 
Talagrand's construction \cite{Talagrand} (also see \cite[Corollary A.2]{Hansen})  shows that this (two-parameter) maximal function can not be bounded on any $L^p$, $p\neq \infty$.

The following is our first result, which is somewhat surprising in that 
the two-parameter maximal function $\mathcal {M}$ has the same $L^p$ boundedness as  the one-parameter maximal function $f\to \sup_{0<t } |\mathcal A_{t}^{c_0t}f|$. 

\begin{thm}\label{global maximal} 
The maximal operator $\mathcal {M}$ is bounded on $L^p$ if and only if $p>2$.
\end{thm}

\subsubsection*{Localized maximal function}  
The localized spherical and circular maximal functions which are defined by taking supremum over radii contained in a compact interval included in 
$(0,\infty)$ have $L^p$ improving property, that is to say, 
the maximal operators are bounded from 
$L^p$ to $L^q$ for some $p< q$. Schlag \cite{Schlag97} and  Schalg--Sogge \cite{SS} characterized the almost complete typeset of $p,q$ except the endpoint cases. 
One of the authors \cite{L} obtained most of the remaining endpoint cases.   There are also results in which  dilation parameter sets were generalized to  sets of fractal dimensions (for example, see \cite{ajs,   sww}). 

In analogue to those results  concerning the localized spherical and circular maximal operators,  
it is natural to investigate  $L^p$-improving property of $\mathcal M_c$ which  is defined by 
 \[
\mathcal {M}_{c}f({x})=\sup_{(t,s)\in\mathbb J}\big|\mathcal A_t^{s}f({x})\big|.
\]
Here $\mathbb J$ is a compact subset of $ \mathbb J_\ast:=\{ (t,s)\in \mathbb R^2: 0<s<t\}.$ 
The next theorem gives $L^p$--$L^q$ bound on $\mathcal {M}_{c}$ on a sharp large of $p,q$. 

\begin{thm}\label{main thm}  
Set  $P_1=(5/11,2/11)$ and $P_2=(3/7,1/7)$. Let $\mathcal{Q}$  be the open quadrangle with vertices $(0,0)$,  $(1/2, 1/2)$,  $P_1$, and $P_2$ which includes the half open line segment $[(0,0), (1/2, 1/2))$. {\rm (See Figure \ref{fig1}.)} Then, the estimate 
\begin{equation}\label{main est}
\Vert \mathcal {M}_c f\Vert_{L^q}\lesssim \Vert f\Vert_{L^p}
\end{equation}
holds  if $(1/p,1/q)\in \mathcal{Q}$. 
\end{thm}

Conversely, if $(1/p,1/q)\notin \overline{\mathcal{Q}}\setminus\{(1/2, 1/2)\}$, then the estimate \eqref{main est} fails.

\begin{figure}
\begin{tikzpicture}[scale=0.7]
\draw[thick] (0,0)rectangle(20/3,20/3);
\fill[black!20!white] (0,0)--(10/3,10/3)--(100/33,40/33)--(20/7,20/21);
\draw (0,0)--(20/3,20/3);
\draw [densely dotted](20/3,0)--(0,20/3);
\draw [densely dotted] (20/3,0)--(5/3,5/3);
\draw [thick](4/3,5/3) node[above]{$(\frac14,\frac14)$};
\draw[thick] (10/3,10/3)--(100/33,40/33);
\draw[thick] (100/33,40/33)--(20/7,20/21);
\draw[thick] (0,0)--(20/7,20/21);
\draw (20/3,0) node[below]{$\frac{1}{p}$};
\draw (0,20/3) node[left]{$\frac{1}{q}$};
\draw (-2/12,0) node[font=\fontsize{10}{10}, below]{$(0,0)$};
\draw (97/33,45/33) node[font=\fontsize{10}{10}\selectfont, right]{$P_1$};
\draw (20/7,20/21) node[below]{$P_2$};
\begin{scope}[every path/.style={->}]
   \draw (0,0) -- (7,0);
   \draw (0,0) -- (0,7);
\end{scope}
\end{tikzpicture}
\label{fig1}
\captionof{figure}{The typeset of $\mathcal M_c$}
\end{figure}
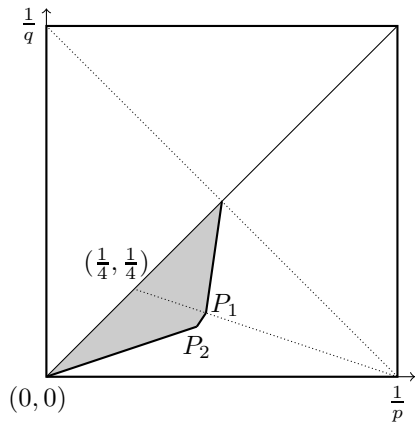

\subsection*{Smoothing estimates for $\cA_t^s$}  Smoothing estimates for averaging operators have a close connection to the associated maximal functions. Especially, the
local smoothing estimate for the wave  operator was used by Mockenhaupt-- Seeger--Sogge \cite{MSS2} to provide  
an alternative proof of the circular maximal theorem. Recent progress \cite{KLO, BGHS, KLO1} on the maximal functions associated with 
the curves in higher dimensions were also achieved  by relying on local smoothing estimates  (also see \cite{PS}).  Analogously, 
our proofs of Theorem \ref{global maximal} and  \ref{main thm} are also based on 2-parameter local smoothing estimates for the averaging operator  $\mathcal A_t^{s}$, which are of independent interest.  
In the following, we obtain the sharp two-parameter local smoothing estimate for $\mathcal A_t^{s}$.

\begin{thm}\label{two para smoothing}   Let $p\ge 2$ and $\psi$ be a smooth function with its support contained in  $\mathbb J_\ast$.  Set  $\tilde \cA_t^s f(x) =\psi(t,s) \cA_t^s f(x).$  
Then, the estimate 
\Be \label{2p-sm}  \| \tilde \cA_t^{s} f\|_{L^p_\alpha(\mathbb R^5)} \lesssim \|f\|_{L^p(\mathbb R^3)}  \Ee 
holds if $\alpha<\min\lbrace 1/{2},4/{p} \rbrace$.  
\end{thm}

The result in Theorem \ref{two para smoothing} is sharp in that $\tilde \cA_t^{s}$ can  not be bounded from $L^p$ to $L^p_\alpha$ if $\alpha>\min\lbrace 1/2,4/p\rbrace$ (see Section \ref{counterexample section} below).  Using 
the estimate \eqref{2p-sm}, one can deduce results concerning the dimension of a union of the 
torus $x+\Gamma_t^s$, $(x,t,s)\in E\subset \mathbb R^3 \times \mathbb J_\ast $. See \cite{HKLO}.

We also obtain the sharp local smoothing estimate for the 1-parameter operator $f\to \cA_{t}^{c_0t}f$. 

\begin{thm}\label{smoothing} Let $\chi_0 \in C_c^\infty(0,\infty)$.  Let $p\geq 2$ and $0<c_0<1$. Then, for $\alpha<\min\lbrace 1/{2},3/{p}\rbrace$, we have
\Be\label{1p-sm} \| \chi_0(t) \cA_t^{c_0t} f\|_{L^p_\alpha(\mathbb R^4)} \lesssim \|f\|_{L^p(\mathbb R^3)}. \Ee
\end{thm}

The estimate above is sharp  since $f\to  \chi_0(t) {\cA}_t^{c_0t} f$ fails to be bounded from $L^p_{{x}}$ to $L^{p}_{\alpha}(\R^4)$
 if $\alpha>\min\lbrace1/{2},3/{p}\rbrace$  (Section \ref{counterexample section}).  The next theorem gives 
 the sharp regularity estimate for $ \cA_t^{s}$ when $s,t$ fixed.

\begin{figure}
\begin{tikzpicture}[scale=0.8]
\fill[black!10!white] (0,0)--(20/3,0)--(20/3,10/3)--(10/6,10/3);
\draw [thick] (0,0)--(10/3,10/3);
\draw (10/6,10/3)--(20/3,10/3);
\draw[dotted] (20/3,10/3)--(20/3,0);
\draw[dotted] (10/3,10/3)--(10/3,0);
\draw (10/3,0) node[below]{$\frac{1}{4}$};
\draw (20/3,0) node[below]{$\frac{1}{2}$};
\draw [thick] (0,0)--(20/9,10/3);
\draw [thick] (0,0)--(15/9,10/3);
\draw[dotted] (10/6,10/3)--(10/6,0);
\draw[dotted] (20/9,10/3)--(20/9,0);
\draw (10/6,0) node[below]{$\frac{1}{8}$};
\draw (20/9,0) node[below]{$\frac{1}{6}$};
\draw[dotted] (0,10/3)--(10/6,10/3);
\draw (0,10/3) node[left]{$\frac{1}{2}$};
\draw (7,0) node[right]{$\frac{1}{p}$};
\draw (0,12/3) node[above]{$\alpha$};
\begin{scope}[every path/.style={->}]
   \draw (0,0) -- (7,0);
   \draw (0,0) -- (0,12/3);
\end{scope}
\end{tikzpicture}
\captionof{figure}{Smoothing orders of the estimates \eqref{2p-sm}, \eqref{1p-sm}, and \eqref{eq:fixed}}
\label{fig2}
\end{figure}
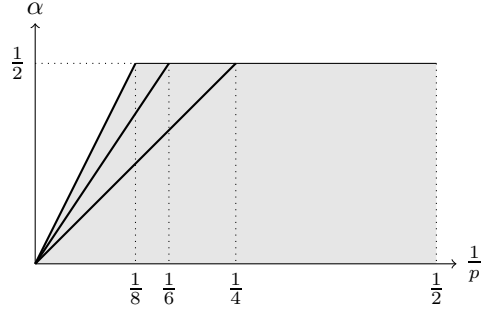

\begin{thm} 
\label{fixed}
Let $0<s<t$. If $\alpha<\min\lbrace 1/{2},2/{p} \rbrace$, then we have 
 \Be\label{eq:fixed}
\|{\mathcal A}_t^{s} f\|_{L^{p}_{\alpha}(\mathbb R^3)}\lesssim  \|f\|_{L^p(\mathbb R^3)}.
 \Ee
\end{thm}

 If $\alpha> \min\lbrace 1/{2},2/{p} \rbrace$,  $\tilde{\mathcal A}_t^{s}$ is not bounded from $L^p(\mathbb R^3)$ to $L^{p}_{\alpha}(\mathbb R^3)$ (Section \ref{counterexample section}). One can compare 
 the local smoothing estimates in  Theorem \ref{two para smoothing} and \ref{smoothing} with the 
 regularity estimate in Theorem \ref{fixed}.  The 2-parameter and 1-parameter local smoothing estimates have extra smoothing of order up to 
 $2/p$ and $1/p$, respectively,  when $p>8$ (see Figure \ref{fig2}).   
 
 For $p< 2$, it is easy to show that there is no additional  smoothing (local smoothing)  for the operators $\tilde{\mathcal A}_t^{s}$ and $\chi_0(t){\mathcal A}_t^{c_0t}$ when 
 compared with the estimates with fixed $s,t$ (Theorem \ref{fixed}). That is to say,    $\tilde{\mathcal A}_t^{s}$ fails to be bounded from $L^p(\mathbb R^3)$ to $L^p_\alpha(\mathbb R^5)$ and so does  $\chi_0(t){\mathcal A}_t^{c_0t}$ from $L^p(\mathbb R^3)$ to $L^p_\alpha(\mathbb R^4)$ if 
 $\alpha> \min (2/p',1/2)$ and $1\le p\le 2$. 
 
 Finally, we remark that our result for two parameter 2-dimensional tori can be extended to 
 multiparameter tori in higher dimensions. We will address this issue elsewhere. 

\subsubsection*{Organization of the paper.}  In Section \ref{sec2}, we obtain various preparatory estimates for the functions which are localized in the Fourier side. 
In Section \ref{sec3}  we prove  Theorem \ref{global maximal}, \ref{main thm},  and \ref{two para smoothing}. The proofs of Theorem \ref{smoothing} and \ref{fixed} are given in Section \ref{sec4}. 
Sharpness of the range of  $p,q$ in  Theorem \ref{main thm} and the smoothing orders in Theorem \ref{two para smoothing},  \ref{smoothing}, and \ref{fixed} are shown in Section \ref{counterexample section}. 

\subsection*{Notation}   We denote ${x}=(\bar x,x_3)\in \mathbb R^2\times \mathbb R$ and similarly  ${\xi}=(\bar\xi,\xi_3)\in \mathbb R^2\times \mathbb R$.  
In addition to $\,\widehat{}\,$ and ${}^\vee$, we occasionally use $\mathcal F$ and $\mathcal F^{-1}$ to denote the Fourier  and inverse Fourier transforms, respectively. 
For two given nonnegative quantity $A$ and $B$, we write $A\lesssim B$ if there is a constant $ C > 0$ such that
$B\le CA$.

\section{Local smoothing Estimates for $\mathcal A_{s,t}$}\label{sec2}
In this section we are mainly concerned with estimates for the averaging operator under frequency localization of the input function. 
We obtain the estimates by making use of  the decoupling inequality and the local smoothing estimate for the wave operator.

We denote 
\[{\mathbb A}_\lambda=   \{  \eta\in \R^2:  2^{-1}\lambda \le   |\eta|\le 2\lambda\}, \quad \bA_\lambda^\circ=   \{  \eta\in \R^2:     |\eta|\le 2\lambda\},\] 
respectively. Similarly, we   set $\mathbb I=[1,2]$ and $\mathbb I^\circ=[0,2]$, and we denote   $\mathbb I_\tau=\tau \mathbb I$ and $\mathbb I_\tau^\circ=\tau \mathbb I^\circ$ for $\tau\in (0,1]$.  

We now consider  the $2$-d wave operator 
\[
\mathcal W_\pm g(y,t)=\frac{1}{(2\pi)^2}\int_{\R^2}e^{i(y\cdot\eta\pm t|\eta|)}\widehat{g}(\eta)d\eta. 
\]
The following  is a consequence of 
the sharp local smoothing due to   Guth--Wang--Zhang \cite{GWZ} (also see \cite{SS}).

\begin{thm}\label{prop:locals}
Let  $2\le p\leq q$, $1/{p}+3/{q}\le  1$, and $\lambda \ge 1$. Then,  the estimate   
\begin{equation}\label{locals-est}
\big\Vert \mathcal W_\pm g \big\Vert_{L^q(\R^2\times \mathbb I^\circ)}\le C
\lambda^{( \frac{1}{2}+\frac{1}{p}-\frac{3}{q} )+\epsilon }\Vert g\Vert_{L^p} 
\end{equation}
holds for any $\epsilon>0$ whenever $\supp \widehat g\subset \bA_\lambda$.  
\end{thm}

\begin{proof} It is sufficient to show the estimate for $\mathcal W_+$ since that for $\mathcal W_-$  follows by conjugation and reflection.  When the interval $\bI^\circ$ is replaced by $\mathbb I$, the desired estimate follows from the  known estimates  and interpolation. 
Indeed, for $1\le p\le q\le \infty$ and $ 1/p+3/q\le 1$,  we have   
\Be
\label{o-locals}
\big\Vert \mathcal W_+ g \big\Vert_{L^q(\R^2\times \mathbb I)}\le C \lambda^{\frac{1}{2}+\frac{1}{p}-\frac{3}{q}+\epsilon }\Vert g\Vert_{L^p} 
\Ee
whenever $\supp \widehat g\subset \bA_\lambda$. This is a consequence of interpolation between the sharp $L^p$ local smoothing  estimates for  $p=q\ge 4$  (\cite{GWZ}) and the estimate $\Vert \mathcal W_+ g \Vert_{L^\infty(\R^2\times \mathbb I)}\le C \lambda^{\frac{3}{2}}\Vert g\Vert_{L^1}$  (e.g., see \cite{SS}).  

By   dyadically  decomposing $\bI^\circ$ away from $0$ and scaling,   one can deduce \eqref{locals-est} from \eqref{o-locals}. 
Indeed, since  
\Be 
\label{scaling} 
\mathcal W_+ g (x,\tau t)=\mathcal W_+ g(\tau\, \cdot ) ( x/\tau, t), 
\Ee
   rescaling gives the estimate 
\[ \big\Vert \mathcal W_+ g \big\Vert_{L^q(\R^2\times \mathbb I_\tau)}\le C\tau^{\frac12-\frac1p} \lambda^{ \frac{1}{2}+\frac{1}{p}-\frac{3}{q}+\epsilon } \Vert g\Vert_{L^p} 
\]
for any $\epsilon>0$ 
if $\supp \widehat g\subset \bA_\lambda$ and $\tau\lambda\gtrsim 1$.  When $\tau\sim \lambda^{-1}$, by 
scaling and an easy estimate we also have $\Vert \mathcal W_+ g \big\Vert_{L^q(\R^2\times \mathbb I_\tau^\circ)}\lesssim \lambda^{2/p-3/q} \|g\|_p$. Now, since $p\ge 2$, decomposing $\bI^\circ=(\bigcup_{\tau\ge (2 \lambda)^{-1}} \bI_\tau^\circ) \cup   \mathbb I_{\lambda^{-1}}^\circ$ and taking sum over the intervals,   we get 
\[ 
\big\Vert \mathcal W_+ g \big\Vert_{L^q(\R^2\times \mathbb I^\circ)}\le C
\max\lbrace\lambda^{ \frac{1}{2}+\frac{1}{p}-\frac{3}{q} +\epsilon }, \lambda^{\frac{2}{p}-\frac{3}{q} } \rbrace \Vert g\Vert_{L^p} 
\lesssim  \lambda^{ \frac{1}{2}+\frac{1}{p}-\frac{3}{q} +\epsilon }  \Vert g\Vert_{L^p} 
\] 
for any $\epsilon>0$.  
\end{proof} 

As a consequence of Theorem \ref{prop:locals} 
we also have the next lemma, which we use  later  to obtain estimate for functions whose Fourier supports are included  in a conical region with a small angle. 

\begin{lem}
\label{lem:locals}
Let $2\le p\le q\le \infty$, $1/{p}+3/{q}\le 1$, and $\lambda\ge 1$. Suppose that $\lambda\lesssim  h \lesssim   \lambda^2$. Then, for any $\epsilon>0$  there is a constant $C$ such that 
\Be \label{locals}
\big\|
\mathcal W_\pm g
\big\|_{L^q (\R^2\times \mathbb I^\circ)}
\le C
\lambda^{1-\frac1p-\frac3q} h^{\frac2p-\frac12 +\epsilon}  \Vert g\Vert_{L^p}  
\Ee
 whenever $\supp \widehat g\subset    \mathbb I_h \times \mathbb I_\lambda^\circ$. 
\end{lem}
\begin{proof}
As before, it is sufficient to consider $\mathcal W_+$. 
By interpolation we only need to check the estimate \eqref{locals} for $(p,q)=(4,4)$, $(2,6)$, $(2,\infty)$, and $(\infty,\infty)$. 
 Since $\lambda\le h$, $\supp \widehat g\subset \{\eta: |\eta|\sim h\}$.  So,  the estimate \eqref{locals}  for $(p,q)=(4,4)$, $(2,6)$, and $(2,\infty)$ is clear from   \eqref{locals-est}.  
  Thus, it suffices to verify  \eqref{locals} for $p=q=\infty$, that is to say, 
  \[ \|
\mathcal W_+ g
\|_{L^\infty(\R^2\times \mathbb I^\circ)}
\lesssim 
\lambda  h^{-1/2}  \Vert g\Vert_{L^\infty}\] 
whenever $\supp \widehat g\subset    \mathbb I_h \times \mathbb I_\lambda^\circ$.  To show this, we cover $\mathbb I_h \times \mathbb I_\lambda^\circ$ by  as many as  $C\lambda h^{-1/2}$ boundedly overlapping  rectangles of dimension $h \times h^{1/2}$  whose principal axis contains the origin, and consider a partition of unity $\{\tilde \omega _\nu\}$  subordinated to those rectangles such that   $(\alpha,\beta)$-th
derivatives of $\tilde \omega_\nu$ in the directions of the principal and its normal directions is  bounded by $Ch^{-\alpha}h^{-\beta/2}$. 
(In fact, one can also use $\omega_\nu(\eta)$ in the proof of Proposition \ref{flocal} below replacing $\lambda$ by $h$.) 
Consequently, we have  $\mathcal W_+  g =\sum_\nu \mathcal W_+  \chi _\nu(D) g$. It is easy to see  that the kernel of  the operator $g\to \mathcal W_+  \chi _\nu(D) g$  has a uniformly bounded $L^1$ norm for $t\in \bI^\circ, \nu$. Therefore, we get the desired estimate.
\end{proof}

\subsection{Two-parameter propagator}
\label{sec:2-p}
We define an operator $\mathcal U$ by 
\[
\mathcal{U} f({x},t,s)=
\int e^{i({x}\cdot{\xi}+t|\bar\xi|+s|{\xi}|)}\widehat{f}({\xi})d{\xi},
\] 
 is closely related to the averaging operator $\cA_t^s$ and the wave  operator $\mathcal W_+$. In fact, we obtain the estimates for $\mathcal U$  making use of 
those for $\mathcal W_+$.

Let $\mathbb J_0=\{ (t,s): 0<s<c_0 t\}$ and $\mathbb J_\tau=(\mathbb I\times \mathbb I_\tau)\cap \mathbb J_0$. To obtain the required  estimates  for our purpose, we consider the estimates over  $\mathbb R^3\times \mathbb J_\tau$ for small $\tau$.

\begin{prop}\label{flocal}
Let $2\leq p\leq q\leq \infty$ satisfy ${1}/{p}+{3}/{q}\le 1$, and let   $0<\tau\le 1$ and $\lambda\ge \tau^{-1} $. 
$(a)$ If  $ \lambda \lesssim h   \lesssim  \tau \lambda^2 $, then  for any $\epsilon>0$ the estimate 
\begin{align}
\label{plocal1}
\Vert \mathcal{U} f\Vert_{L^q(\R^3\times \mathbb J_\tau)}\lesssim 
\tau^{(\frac{1}{2}-\frac{1}{p})}   \lambda^{\frac{3}{2}-\frac{1}{p}-\frac{5}{q}}  h^{-\frac12+\frac2p+\epsilon} \Vert f\Vert_{L^p} \end{align}
holds whenever $\supp \widehat f \subset  \mathbb A_\lambda\times \mathbb I_h$. 
Moreover, $(b)$ if $\supp \widehat f \subset  \mathbb A_\lambda\times \mathbb I_\lambda^\circ$, then we have the estimate 
\eqref{plocal1} with $h=\lambda$. $(c)$ If $h  \gtrsim  \tau \lambda^2,$ then we have 
\begin{align}
\label{plocal2}
\Vert \mathcal{U} f\Vert_{L^q(\R^3\times \mathbb J_\tau)}\lesssim  \tau^{\frac1q} \lambda^{\frac{1}{2}+\frac{1}{p}-\frac{3}{q}+\epsilon } h^{\frac1p-\frac1q} \Vert f\Vert_{L^p}
\end{align}
whenever $\supp \widehat f \subset  \mathbb A_\lambda\times \mathbb I_h$. 
\end{prop}

For a bounded measurable function $m$, we denote by $m(D)$ the multiplier operator 
defined by $\mathcal F(m(D)f)(\xi)= m(\xi)\widehat f(\xi)$. In what follows, we occasionally use the following lemma.  

\begin{lem}\label{fff} Let $\xi=(\xi', \xi'')\in \mathbb R^{k}\times \mathbb R^{d-k}$. Let $\chi$ be an integrable  function on $\mathbb R^{k}$  such that $\widehat \chi$ is also integrable. 
Suppose $\|m(D)f\|_q\le B\|f\|_p$ for a constant $B>0$, then we have
$\|m (D)\chi(D')f\|_q\le B \|\widehat\chi \|_1 \|f\|_p$. 
\end{lem}

This lemma  follows from  the identity 
\[  m (D)\chi(D')f(x) = (2\pi)^{-k} \int_{\mathbb R^k}   \widehat \chi(y)(m(D) f)(x'+y, x'') dy,\] 
which is a simple consequence of the Fourier inversion.  The desired inequality is immediate   from Minkowski's inequality and translation invariance of $L^p$ norm.

\begin{proof}[Proof of Proposition \ref{flocal}] 
We make use of the decoupling inequality  for the cone \cite{BD}  and the sharp local smoothing estimate (Lemma \ref{lem:locals}) for 
$\mathcal W_+$. 

We first show the case $(a)$ where $ \lambda \lesssim h   \lesssim  \tau \lambda^2 $. To this end, we prove the estimate \eqref{plocal1} under the additional assumption that $q\geq 6$. 
We subsequently extend the range by interpolation between the consequent estimates and  \eqref{plocal1} for $(p,q)=(4,4)$, which we prove later. 

Fixing $x_3$ and $s$,  we define an operator $\mathcal{T}_{x_3}^s$ by setting 
\[ \widehat{\mathcal{T}_{x_3}^s F}(\bar \xi)=\int e^{i(x_3\xi_3+s|{\xi}|)}\widehat{F}(\bar\xi, \xi_3)d\xi_3, \quad \xi=(\bxi, \xi_3). \]
Then, we observe that 
\[ \mathcal{U} f({x},t,s)=\mathcal{W}  (\mathcal{T}_{x_3,s}f)(\bar x,t). \]

Let $\mathfrak{V}_\lambda\subset \mathbb S^1 $ be a collection of   $\sim \lambda^{-1/2}$-separated points. By $\{ w_\nu\}_{\nu\in \mathfrak{V}_\lambda}$ we  denote  a partition of unity on the unit circle  $\mathbb S^1$ such that 
$w_\nu$ is supported in an arc centered at $\nu$ of length about $\lambda^{-1/2}$ and $|(d/d\theta)^k w_\nu|\lesssim \lambda^{k/2}$. 
For each $\nu\in\mathfrak{V}_\lambda$, we set $\omega_\nu(\bxi)=w_\nu(\bxi/|\bxi|)$ and  
\[ \mathcal{W}_\nu g(\bx,t)= \int e^{i(\bx\cdot\bxi+t|\bxi|)} \omega_{\nu}(\bxi) \widehat{g}(\bxi)d\bxi.\]

Let $\tilde\chi \in \mathcal S(\mathbb R)$ such that $\tchi\ge 1$ on $\bI$ and $\supp \cf (\tchi)\subset [-1/2,1/2]$. 
Note that  the Fourier transform of $\tchi(t)\mathcal{W}_\nu g(\bx,t)$ is supported in 
the set $\{(\bxi, \tau) :  |\tau-|\bxi||\lesssim 1, \bxi/|\bxi|\in\supp \omega_\nu, |\bxi|\sim \lambda  \}$ if $\supp \widehat g\subset \bA_\lambda$. Thus, by Bourgain--Demeter's $l^2$ decoupling inequality \cite{BD} followed by H\"older's inequality, we have 
\Be 
\label{BD-}
\big\|   \sum_{\nu\in\mathfrak{V}_\lambda}\mathcal{W}_\nu g\big\|_{L^q_{\bx, t}(\mathbb R^2\times \mathbb I)}
\lesssim  \lambda^{\frac{1}{2}-\frac1{2p}-\frac{3}{2q}+\epsilon } 
\Big( \sum_{\nu\in\mathfrak{V}_\lambda}\big\| \tchi(t)\mathcal{W}_\nu g\big\|_{L^q_{\bx,t} (\mathbb R^3)}^p \Big)^{1/{p}}   
\Ee
for any  $\epsilon>0$, $q\ge 6$, and $p\ge 2$,  provided that $\supp \widehat g\subset \bA_\lambda$.  
Note that $\mathcal{U} f({x},t,s)=\sum_\nu\mathcal{W}_\nu  (\mathcal{T}_{x_3}^sf)(\bar x,t)$ and $\mathcal{W}_\nu  (\mathcal{T}_{x_3}^sf)(\bar x,t)=\mathcal{U} \omega_\nu(\bar D) f({x},t,s)$. 
Since $\supp \widehat f \subset  \mathbb A_\lambda\times \mathbb I_h$,  freezing $s, x_3$,  
we apply  the inequality \eqref{BD-}, followed by Minkowski's inequality,  to get 
\Be
\label{decouple-}
\Vert \mathcal{U} f \Vert_{L^q(\R^3\times \mathbb J_\tau) } \lesssim  \lambda^{\frac{1}{2}-\frac1{2p}-\frac{3}{2q}+\epsilon }
\Big( \sum_{\nu\in\mathfrak{V}_\lambda}\big\| \tchi(t) \mathcal{U} f_\nu \big\|_{L^q_{x,t,s} (\R^4\times  \mathbb I_\tau)}^p \Big)^{1/{p}}
\Ee
for $q\geq 6$ where $f_\nu=\omega_\nu(\bar D)  f$. We now claim that 
\Be 
\label{claim1-1}
\Vert  \tchi(t) \mathcal{U} f_\nu \Vert_{L^q(\R^4\times \mathbb I_\tau) } \lesssim 
\tau^{(\frac{1}{2}-\frac{1}{p})}   \lambda^{1-\frac{1}{2p}-\frac{7}{2q}}  h^{\frac2p-\frac12+\epsilon} \Vert f_\nu\Vert_{L^p} 
\Ee
holds for $1/p+3/q\le 1$.   Note that $ (\sum_\nu \|f_\nu\|^p_p)^{1/p}$ for $1\le p\le \infty$. Thus,  from \eqref{decouple-} and \eqref{claim1-1}  the estimate \eqref{plocal1}  follows for $q\ge 6$.

To obtain \eqref{claim1-1}, we begin by showing 
\Be 
\label{claim2-1}
\big\| \tchi(t) \mathcal{U} f_\nu (\cdot,s)\big\|_{L^q_{x,t} (\mathbb R^4)}\le C\big\|  e^{is|D|} f_\nu\big\|_{L^q_{x} (\mathbb R^3)}.
\Ee
To do this, we apply the argument used to show Lemma \ref{fff}.  Let us set 
\[  \tilde  \chi_\nu(t,\bxi) =   e^{it(|\bxi|-\bxi\cdot\nu)} \widetilde \omega_\nu (\bxi)  \varphi(\bxi/\lambda)  \]
so that $\tilde  \chi_\nu(t,\bxi) \widehat f_\nu(\xi)=  e^{it(|\bxi|-\bxi\cdot\nu)} \widehat f_\nu(\xi)$. Here $\widetilde \omega_\nu (\bxi) $ is a angular cutoff function given in the same manner as $\omega_\nu (\bxi)$
such that $\widetilde \omega_\nu \omega_\nu=\omega_\nu$. 
Then, a computation shows that 
\[ |( \nu\cdot\nabla_\bxi )^k  (\nu_\perp \cdot \nabla_\bxi)^l \tilde  \chi_{\nu} (t,\bxi)| \lesssim   (1+|t|)^{k+l}\lambda^{-k}\lambda^{-\frac l2}(1+ \lambda^{-1} |\nu\cdot \bxi|)^{-N} (1+  \lambda^{-\frac12}|\nu^\perp\cdot \bxi|)^{-N}\]
for any  $N$ where $\nu_\perp $ denotes a unit vector orthogonal to $\nu$. Indeed, this can be easily seen via rotation and scaling (i.e., setting $\nu=e_1$ and scaling $\xi_1\to \lambda \xi_1$ and $\xi_2\to \lambda^{1/2} \xi_2$). Thus, using the above inequality for $0\le k, l\le 2$ and integration by parts, we see   $\|( \tilde \chi_\nu(t, \cdot))^\vee\|_1\le  C(1+|t|)^{4}$ for a constant $C>0$. Since $\mathcal{U} f_\nu(x,t,s)=\mathcal F^{-1}( e^{i(t\nu\cdot{\bxi}  +s|\xi|)} \tilde  \chi_{\nu} (t,\bxi)\widehat{f}_\nu({\xi}))$,  we have
\[ \mathcal{U} f_\nu(x,t,s)= \int  ( \tilde \chi_{\nu}(t,\cdot))^\vee(\eta)\, e^{is|D|} f_\nu(\bar x-\eta+t\nu,x_3) d\eta.\]
By Minkowski's inequality and changing variables $\bx\to \bar x+\eta-t\nu$ we see that 
the left hand side of \eqref{claim2-1} is bounded by 
$C \| \tchi(t) (1+|t|)^4 \|_{L^q_t(\mathbb R^1)} \big\|  e^{is|D|} f_\nu\big\|_{L^q_{x} (\mathbb R^3)}$.  
Therefore, we get the desired inequality \eqref{claim2-1}.

Let us set 
\[  \chi_s (\xi) = e^{is(|\xi|-|\xi^\nu|)} \widetilde \omega_\nu (\bxi)  \varphi(\bxi/\lambda) \varphi(\xi_3/h),  \] 
where $\xi^\nu:=(\bxi\cdot\nu, \xi_3)$.  Since $\lambda\lesssim h$, similarly as before,  one can easily see $\| \widehat \chi_s\|_1\le C$ for a constant. Thus, by Lemma \ref{fff} we have $\|  e^{is|D|} f_\nu \|_{L^q_x }\lesssim  \|  e^{is|\bar D^\nu|} f_\nu \|_{L^q_x }  $.  Combining this and \eqref{claim2-1} yields 
\[  \big\| \mathcal{U} f_\nu \big\|_{L^q(\R^3\times \mathbb J_\tau)} \lesssim    \|e^{is|\bar D^\nu|} f_\nu \|_{L^q_{x,s}(\mathbb R^3\times \bI_\tau)} 
\lesssim   \lambda^{\frac1{2p}-\frac1{2q}}  \|e^{is|\bar D^\nu|} f_\nu \|_{L^p_{\bx_\nu'} (L^q_{ \bx_\nu, x_3, s}(\mathbb R^2\times \bI_\tau))},   \] 
where $\bx_\nu=\nu\cdot \bx$ and  $\bx_\nu'=\nu_\perp \cdot \bx$. 
For the second inequality we use  Bernstein's inequality (see, for example, \cite[Ch.5]{Wolff}) and Minkowski's inequality together with the fact that the projection of $\supp\widehat f$ to 
$\sspan \{\nu_\perp\}$ is contained in an interval of length $\lesssim \lambda^{1/2}$.

Note  that the projection  $\supp\widehat f $ to  $\sspan\{  \nu, e_3\}$ is contained in the rectangle $\bI_\lambda\times \bI_h$.  
By rotation  the matter is reduced to obtaining estimate for the 2-d wave operator. That is to say,  
the inequality  \eqref{claim1-1}  follows for $q\ge 6$ if we show 
\[
\big\|
\mathcal W_+ g
\big\|_{L^q (\R^2\times \mathbb I_\tau)}
\lesssim 
\tau^{\frac12-\frac1p} \lambda^{1-\frac1p-\frac3q} h^{\frac2p-\frac12 +\epsilon}  \Vert g\Vert_{L^p}   
\] 
for $1/p+3/q\le 1$ whenever $\supp \widehat g \subset \bI_{h}\times \bI_{\lambda}^\circ$.  This  inequality 
is an immediate consequence of \eqref{locals} and scaling. Indeed, as before,  after scaling (i.e., \eqref{scaling})
we  apply Lemma \ref{locals} with $\supp \cf (g(\tau\,\cdot)) \subset \bI_{\tau h}\times \bI_{\tau \lambda}^\circ$. 
To this end, we use  the  condition $ h\le \tau\lambda^2$, equivalently, 
$ \tau h\le (\tau\lambda)^2$.  

 We now have  the estimate  \eqref{plocal1}  for $6\le q$, $2\le p$,  and $1/p+3/q\le 1$. In order to prove  it in the full range, by interpolation we only have to show  \eqref{plocal1}  for $p=q=4$. 

Let us define $f_\pm$ by setting $\widehat f_\pm(\xi) = \chi_{(0, \infty)}(\pm \xi_2) \widehat f(\xi)$ where $\chi_E$ denotes the character function of a set $E$.  Then, changing variables  $\xi_2\to \pm\sqrt{\rho^2- \xi_1^2}$, we write  
\[ \mathcal{U}f({x},t,s)= \textstyle \sum_{\pm} \int e^{i(x_3\xi_3+t\rho+ s\sqrt{\rho^2+\xi_3^2})}\cf(\mathcal{S}_{\pm}^{\bx} f_\pm)(\rho,\xi_3) d\rho d\xi_3,  \]
where
\[ \mathcal F(\mathcal{S}_{\pm}^{\bx} f_\pm)( \rho,\xi_3)=\pm \int e^{i(x_1\xi_1\pm x_2\sqrt{\rho^2-\xi_1^2})}\widehat{f_\pm}(\xi_1, \pm \sqrt{\rho^2-\xi_1^2},\xi_3)\frac{\rho}{\sqrt{\rho^2-\xi_1^2}}\,d\xi_1. \]

 We  observe the following, which is a consequence of the estimate \eqref{o-locals} with 
$p=q=4$ and   the finite speed of propagation of the wave operator: 
\Be 
\label{fs} \| \mathcal  W_+  g  \Vert_{L^4_{{x_3},t,s}(\mathbb R \times \bI\times \mathbb I_\tau) }\lesssim  \tau^\frac14(\tau h)^\epsilon 
 \|  g  \Vert_{L^4_{{x_3},t }(\mathbb R \times \bI^\circ_2) }  +  h^{-N} \| t^{-N} g  \Vert_{L^4_{{x_3},t }(\mathbb R \times (\bI^\circ_2)^c )}  
 \Ee
 for any $N$ whenever $\supp g\subset \{ \bxi: |\bxi|\sim h\}$. Indeed, to show this  we  
decompose $g=g_1+g_2:=g \chi_{\bI^\circ_2}(y_2)+  g \chi_{({\bI^\circ_2})^c}(y_2).$  By finite speed of propagation  (in fact, by straightforward kernel estimate) we have  $ \| \mathcal  W_+  g_2  \Vert_{L^4(\mathbb R \times \bI\times \mathbb I_\tau) }\lesssim  h^{-N} \|  |y_2|^{-N} g  \Vert_{L^4(\mathbb R\times ({\bI^\circ_2})^c ) }.$ Meanwhile, by scaling and \eqref{o-locals} with 
$p=q=4$, we have  $ \| \mathcal  W_+  g_1  \Vert_{L^4(\mathbb R \times \bI\times \mathbb I_\tau) }\lesssim  \tau^\frac14(\tau h)^\epsilon 
 \|  g  \Vert_{L^4(\mathbb R \times \bI^\circ_2) } $. Combining those two estimates, we obtain \eqref{fs}.

 We now note that 
 $\mathcal{U}f({x},t,s)=\sum_{\pm} \mathcal W_+(\mathcal{S}_{\pm}^{\bx} f_\pm)(x_3, t,s)$ 
 and 
 $\supp \mathcal F(\mathcal{S}_{\pm}^{\bx}  f_\pm)\subset  \{ \bxi: |\bxi|\sim h\}  $
since $\lambda\le h$. 
 Here, we regard $(x_3,t)$ and $s$ as the spatial and temporal variables, respectively. 
Applying \eqref{fs} to $\mathcal W_+(\mathcal{S}_{\pm}^{\bx}  f_\pm)$ with $g=\mathcal{S}_{\pm}^{\bx}  f_\pm$,  we obtain
\[ \textstyle \Vert \mathcal{U} f \Vert_{L^4_{{x},t,s}(\mathbb R^3\times \mathbb J_\tau)}
\lesssim  \sum_{\pm} \Big(
\tau^{\frac1{4}} h^\epsilon   \Vert \mathcal{S}_\pm^{\bx}  f\Vert_{L^4_{{x},t}(\mathbb R^3 \times \bI^\circ_2)} +h^{-N} \| t^{-N}  \mathcal{S}_\pm^{\bx}  f  \Vert_{L^4_{{x},t }(\mathbb R^3 \times (\bI^\circ_2)^c )} \Big). \]
Reversing the change of variables $\xi_2\to \pm\sqrt{\rho^2- \xi_1^2}$, we note that  $ \mathcal{S}_\pm^{\bx}  f(x_3,t)=\mathcal W_+ f_\pm(\cdot, x_3)(\bx, t)$.  
Recalling $\supp \mathcal F f \subset \bA_\lambda\times \bI_h$, we  see that the second term in the right hand side is bounded by a constant times $h^{-N/2} \| f  \Vert_{L^4}$. 
Since $\supp \mathcal F(f(\cdot, x_3))\subset \bA_\lambda$ for all $x_3$, using  Lemma \ref{lem:locals} for $p=q=4$,  we obtain  \eqref{plocal1}  for $p=q=4$. 
This completes the proof of $(a)$. 

The case $(b)$ in which $\supp \widehat f \subset  \mathbb A_\lambda\times \mathbb I_\lambda^\circ$ can be handled without change. 
We only need to note   that the Fourier support of $f_\nu$ is included in $\{ \xi:  |(\xi\cdot\nu, \xi_3)|\sim \lambda\}$, instead of $\{ \xi:  |(\xi\cdot\nu, \xi_3)|\sim h\}$, if $f_\nu\neq 0$.

We now consider the case $(c)$ where $\supp \widehat f \subset  \mathbb A_\lambda\times \mathbb I_h$ with $\tau\lambda^2\le h$. The estimate \eqref{plocal2}
is easier to show. We note that  the Fourier transform of  
\[ 
 e^{is(|\xi|-|\xi_3|)}  \varphi(\bxi/\lambda) \varphi(\xi_3/h)\]
 has uniformly bounded $L^1$ norm. One can easily verify this using $\partial_{\xi}^\alpha    s(|(\lambda\bxi, h\xi_3)|-|h\xi_3|)=O(1)$ on $\bA_1^\circ\times \bI_1$ if $\tau\lambda^2\le h$. 
 Thus, by Lemma \ref{fff} we have $\| \mathcal  U f (\cdot, t,s) \Vert_{L^q} \lesssim   \| e^{it|\bar D|} f \Vert_{L^q} $ uniformly in $s$. So, taking integration in $t,s$, we get
 \[ \| \mathcal  Uf  \Vert_{L^q (\R^3\times \mathbb J_\tau)} \lesssim  \tau^{\frac1q} \| e^{it|\bar D|} f  \Vert_{L^q (\R^3\times \bI)} \lesssim  \tau^{\frac1q} h^{\frac1p-\frac1q}
  \| e^{it|\bar D|} f  \Vert_{L^p_{x_3} (L^q_{\bx,t} (\R^2\times \bI)) } . \] 
  For the second inequality we use Bernstein's and Minkowski's inequalities. 
  Using Proposition \ref{prop:locals} in $\bx,t$, we obtain the estimate \eqref{plocal2} for $2\leq p\leq q\leq \infty$ satisfying ${1}/{p}+{3}/{q}\le 1$.  
  \end{proof}

\begin{rmk} 
\label{pm} 
 Following the argument in the proof of Proposition \ref{flocal} and using Theorem \ref{prop:locals} and Lemma \ref{lem:locals}, one can see without difficulty that $f \to \mathcal Uf(x,-t,s)$ satisfies 
the same estimates in Proposition \ref{flocal} in place of $\mathcal U$. Then, by conjugation and reflection it follows that the estimates also hold for $f \to \mathcal Uf(x,\pm t,-s)$.
\end{rmk}

\subsection{Estimates for the averaging operator $\mathcal A_t^{s}$}\label{sec2.2} Making use of the estimates for $\mathcal U$ in Section \ref{sec:2-p} (Proposition \ref{flocal}), 
we obtain estimates for the averaging operator  $\mathcal A_t^{s}$ while assuming  the input function is localized in the Fourier side. These estimates are to play crucial roles  in proving Theorem \ref{global maximal}, \ref{main thm}, and 
\ref{two para smoothing}.

We relate  $\cA_t^s$ to $\mathcal U$ via asymptotic expansion of the Fourier transform of $d\sigma_t^s$. Note that 
\Be 
\label{ftmeasure}
\widehat{d\sigma_t^s} (\xi)=\int_0^{2\pi} e^{-is\sin\theta \cdot\xi_3}\widehat{d\mu}((t+s\cos\theta) \bxi\,)d\theta, 
\Ee
where $d\mu$ denotes the normalized  arc length measure on the unit circle. 
We recall  the well known asymptotic  expansion of the Bessel function (for example, see \cite{Stein}):  
\begin{equation}\label{bessel}
\widehat{d\mu}(\bxi)=\sum_{\pm,\, 0\le j\le N} C_j^{\pm}|\bxi|^{-\frac{1}{2}-j}e^{\pm i|\bxi|}+E_N(|\bxi|),\qquad |\bxi|\gtrsim 1
\end{equation}
for some constants $C_j^{\pm}$ where  $E_N$ is a smooth function satisfying
\begin{equation}\label{bessel error decay}
\big| (d/dr)^lE_N(r) \big|\le C r^{-l-(N+1)/4},  \quad  0\le l \le N',
\end{equation}
 for $r\gtrsim 1$ and a constant $C>0$ where $N'=[(N+1)/4]$. We use  \eqref{bessel} by taking  $N$ large enough.

 Combining \eqref{ftmeasure} and \eqref{bessel} gives an asymptotic expansion for $\mathcal F(d\sigma_t^s)$, which we utilize by
  decomposing  $f$ in the Fourier side. We consider the cases  $\supp \widehat f\subset \{\xi: |\bxi|>1/\tau\}$  and $\supp \widehat f\subset\{ \xi: |\bxi|\leq1/\tau\}$, separately. 
 
\subsection{When $ \supp \widehat f \subset \mathbb A_\lambda^\circ\times \mathbb R$, $\lambda\le 1/\tau$} If $\supp \widehat f \subset \mathbb A^\circ_{1/\tau} \times \mathbb   \bI^\circ_{1/\tau},$
the sharp estimates are easy to obtain.

\begin{lem} 
Let  $1\leq p\leq q\leq \infty$ and $\tau\in (0,1]$. Suppose $\supp \widehat f \subset B(0, 1/\tau):=\{x: |x|<1/ \tau\}$. Then, for a constant $C>0$  we have 
\begin{align}
 \label{st00} \Vert \mathcal A_t^{s}f\Vert_{L^q_{{x},t,s}(\mathbb R^3\times \mathbb J_\tau)}\le C \tau^{\frac{4}{q}-\frac{3}{p}} \Vert f\Vert_{L^p}.
 \end{align} 
\end{lem}

\begin{proof} 
Since $\mathcal A_t^{s}$ is a convolution operator and $\supp \widehat f \subset B(0, \tau^{-1})$, Bernstein's inequality gives 
$ \Vert \mathcal A_t^{s} f\Vert_{L^q_{{x}}}\lesssim \tau^{\frac{3}{q}-\frac{3}{p}}\Vert \mathcal A_t^{s} f\Vert_{L^p_x}$ for any $s,t\in \mathbb R$. 
Thus, we have 
\Be\label{00}
 \Vert \mathcal A_t^{s} f\Vert_{L^q_{{x}}}\lesssim \tau^{\frac{3}{q}-\frac{3}{p}}\Vert f\Vert_{L^p}, \quad \forall s,t \in \mathbb R.
 \Ee The inequality \eqref{st00} follows by integration  in $t,s$ over $\mathbb J_\tau$. 
\end{proof}

\begin{prop}
Let  $1\leq p\leq q\leq \infty$,  $\tau\lesssim  1$, and $h\gtrsim 1/\tau$. Suppose $\supp \widehat f \subset \bA_1^\circ\times \bI_h$. 
Then, we have  
\begin{align}
\label{000st}
&\Vert \mathcal A_t^{s} f\Vert_{L^q_{{x},t,s}(\R^3\times \mathbb J_\tau)} \lesssim \tau^{1/q} (\tau h)^{-\frac1{2}} h^{\frac{1}{p}-\frac{1}{q}}\Vert  f\Vert_{L^p}.
\end{align}
\end{prop}

\begin{proof} To prove \eqref{000st} it is sufficient to show, for  a positive constant $C$, 
\Be \label{000}
\Vert \mathcal A_t^{s} f\Vert_{L^q_{{x}}} \le C (\tau h)^{-\frac1{2}} h^{\frac{1}{p}-\frac{1}{q}}\Vert  f\Vert_{L^p}, 
\quad \forall (t,s) \in  \mathbb J_\tau. 
\Ee
In fact, integration over $ \mathbb J_\tau$ yields \eqref{000st}. 

For simplicity, we denote $\mathbf v_{\!\phi} =(\cos\phi, \sin\phi)$, and we note that 
\[ \mathcal A_t^{s}  f(x)= (2\pi)^{-3}\int \int  e^{i((\bx-t\vphi)\cdot\bxi+x_3\xi_3-s(\vphi\cdot\bxi,\xi_3)\cdot \vthe)} \widehat f(\xi) d\phi d\theta d\xi.  \]
Since $\supp \widehat f \subset \bA_1^\circ\times \bI_h$,  we may disregard the factor $e^{-it\vphi\cdot\bxi}$ using Lemma \ref{fff}.
 Indeed, let $\rho\in C_c( \bA_2^\circ)$ such that $\rho=1$ on $ \bA_1$. 
 Setting  $\rho_{t}^\phi (\bxi)=\rho(\bxi)e^{i t\vphi\cdot\bxi}$, we see  $\|\mathcal F(\rho_{t}^\phi)\|_1\le C$ for a constant $C>0$ and $|t|\lesssim 1$. 
Thus, by Minkowski's inequality and Lemma \ref{fff} we have  
\[   \Vert \mathcal A_t^{s}  f\Vert_{L^q_{{x}}}\lesssim \sup_{\phi}   \Big\| \int   e^{i {x}\cdot{\xi}} 
\int_0^{2\pi} e^{-is(\vphi\cdot\xi,\xi_3)\cdot \vthe } d\theta \widehat f(\xi) d \xi  \Big\|_{L^q_{{x}}}\]
for $|t|\lesssim 1$.   We denote ${\xi}_\phi =(\vphi\cdot\xi,\xi_3)$, and notice that $|s{\xi}_\phi|\gtrsim 1$ since  $h\tau\ge 1$. 
So, usng \eqref{bessel}, we have    
\[ \tx \int e^{-is{\xi}_\phi\cdot \vthe}d\theta=\sum_{\pm,\, 0\le j\le N}  C_j^{\pm}|s{\xi}_\phi|^{-\frac{1}{2}-j}e^{\pm is|{\xi}_\phi|}+E_N(s|{\xi}_\phi|).  \]

To show \eqref{000}, we  obtain  only  the estimates for the operators  $m_s^\pm(D)$, $E_N(s|D_\phi|)$ whose multipliers are given by  
\[ m_s^\pm(\xi):=|s{\xi}_\phi|^{-1/{2}}e^{\pm is|{\xi}_\phi|}, \quad E_N(s|{\xi}_\phi|).\]
Contributions from the multiplier operators associated  with the other terms can be handled similarly but those are easier. 
Since $|\bar\xi|<2$ and $|\xi_3|\sim h\geq 1/\tau$, we use the Mikhlin multiplier theorem and Lemma \ref{fff} to see 
\[ 
\big\|  m_s^\pm(D) f\big\|_{L^q_{{x}}}\lesssim (\tau h)^{-\frac12} 
\Big\| \int e^{i({x}\cdot{\xi}\pm s|\xi_3|)} \widehat f ({\xi})d{\xi} \Big\|_{L^q_{{x}}}\le  (\tau h)^{-\frac12} \| f\|_{L^q_{{x}}}. 
\]
Since $\supp \widehat f\subset \bA_1^\circ \times \bI_h$, by Bernstein's lemma  we have $\Vert f \Vert_{L^q}\lesssim h^{\frac{1}{p}-\frac{1}{q}}\Vert f \Vert_{L^p}.$
This gives the desired estimates for $m_s^\pm(D)$. For the multiplier operator $E_N(s|D_\phi|)$, note from \eqref{bessel error decay} that 
$\partial_{\xi_\phi}^\alpha(|s\xi_\phi|^{N'} E_N(|s\xi_\phi|)\le C(|s\xi_\phi|^{-|\alpha|})$ for $|\alpha|\le N'$ and a constant $C>0$. Using the Mikhlin multiplier theorem again, 
we have 
\[ \big\| E_N(s|D_\phi|) f  \big\|_{L^q_{{x}}} \lesssim \Big\| \int e^{i{x}\cdot{\xi}}|s\xi_3|^{-N'}\widehat{f}({\xi})d{\xi}d\theta \Big\|_{L^q_{{x}}}. \] 
Since $\supp \widehat f\subset \bA_1^\circ \times \bI_h$,  we see, as before, that the right hand side is bounded by 
$
C (h\tau)^{-N'}h^{1/{p}-1/{q}}\Vert f\Vert_{L^p}.
$
Thus, the desired estimate for  $E_N(s|D_\phi|)$ follows.  
 \end{proof}

When $\lambda\gtrsim 1$, to handle the case  $\supp \widehat f \subset \bA_\lambda \times \bI_h$ we need more than the estimates with fixed $t,s$. We need the 
smoothing estimates obtained in Section \ref{sec:2-p}.

\begin{prop} 
\label{case20}
Let $2\le p\le q\le \infty$,  $1/p+1/q\le 1$, and $1\lesssim \lambda  \lesssim 1/\tau\lesssim h$.  Suppose $\supp \widehat f \subset \bA_\lambda \times \bI_h$. Then,  for any $\epsilon>0$ we have the following:
\begin{align}\label{0jnk local}
\Vert \mathcal A_t^{s}  f\Vert_{L^q(\R^3\times \mathbb J_\tau)}&\lesssim
\tau^\frac1q (\tau h)^{-\frac12}  h^{\frac{1}{p}-\frac{1}{q}} \lambda^{\frac{1}{p}-\frac{3}{q}+\epsilon}\Vert f\Vert_{L^p},  \quad &&1/{p}+3/{q}\le 1,
\\
\label{0jnk local1}
\Vert \mathcal A_t^{s}  f\Vert_{L^q(\R^3\times \mathbb J_\tau)}&\lesssim
\tau^\frac1q (\tau h)^{-\frac12}  h^{\frac{1}{p}-\frac{1}{q}} \lambda^{-\frac12+  \frac{3}{2p}-\frac{3}{2q}+\epsilon}\Vert f\Vert_{L^p},   \quad && 1/{p}+3/{q}> 1.
\end{align}
\end{prop}

To show Proposition \ref{case20}, as mentioned above,  we  use the asymptotic expansion of the Fourier transform of $d\sigma_t^s$.  Let us set 
\[  m^\pm_l (\xi,t,s)= \int e^{-i (s\xi_3\sin\theta \mp s|\bxi|\cos\theta)} a_l(\theta,t,s) d\theta, \]
where $a_l(\theta,t,s)=(t+s\cos\theta)^{-(2l+1)/2}$. Putting \eqref{ftmeasure} and \eqref{bessel} together, we have 
\Be 
\label{exp}
\tx\widehat{d\sigma_t^s}(\xi)=  \sum_{\pm, 0\le l\le N} \, M_l^\pm(\xi,t,s) +\mathcal E (\xi,t,s)  
\Ee 
for $|\bxi|\gtrsim 1$ where  
\begin{align}
\label{ml}
  M_l^\pm(\xi,t,s)&= C_l |\bxi|^{-l-\frac{1}{2}}   e^{\pm it |\bxi|} m^\pm_l (\xi,t,s), \qquad l=0, \dots, N, 
 \\
 \label{ce}
  \mathcal E (\xi,t,s)&= \int e^{-is\xi_3\sin\theta} E_N((t+s\cos\theta)|\bxi|) d\theta. 
 \end{align}

\begin{proof}
We first show \eqref{0jnk local}. From \eqref{exp}  we need to obtain estimates for the operators  associated to the  multipliers  $M_l^\pm$ and $\mathcal E$. 
The major contributions are from $M_l^\pm(D,t,s)$.  We claim that 
\begin{equation}
\label{tml}
\big\| M_l^\pm(D,t,s) f
\big\|_{L^q_{{x},t}(\mathbb R^3\times \mathbb J_\tau)}\lesssim 
\tau^\frac1q (\tau h)^{-\frac 12}  h^{\frac{1}{p}-\frac{1}{q}} \lambda^{\frac{1}{p}-\frac{3}{q}-l +\epsilon} 
\Vert f\Vert_{L^p}
\end{equation}
holds for  $p\le q$ and ${1}/{p}+{3}/{q}\le 1$. To show this,  we consider the operator $e^{\pm it |\bar D|}  m^\pm_l (D,t,s)$. 
Note that $m^\pm_l (\xi,t,s)= \int e^{-is(\mp |\bxi|,\xi_3)\cdot \vthe}  a_l(\theta,t,s) d\theta $. 
By the stationary phase method, we have 
\begin{equation}\label{bessel analog}
m^\pm_l (\xi,t,s)
 = \sum_{\pm, 0\le j\le N} B_j^{\pm}|s{\xi}|^{-\frac{1}{2}-j}e^{\pm i|s{\xi}|}+\tilde E_N^\pm (s|{\xi}|), \quad (t,s)\in \mathbb J_\tau
\end{equation}
for $|s\xi|\gtrsim 1$. 
Here,  $B_l^{\pm}$  and $\tilde E_N^\pm $  depend on $t,s$.  However,  $(\partial/\partial_\theta)^k a_l $ is uniformly bounded since  $s<c_0t$, i.e.,  $(t,s)\in \mathbb J_0$, so  $B_l^{\pm}$ are uniformly bounded and  $\tilde E_N^\pm$ satisfies \eqref{bessel error decay} in place of $E_N$  as long as   $(t,s)\in \mathbb J_\tau$.

 For the error term $\tilde E_N^\pm(s|{\xi}|) 
$,  we can replace it, similarly as before, by $|s{\xi}|^{-N'}$ using the Mikhlin multiplier theorem. Thus, using \eqref{o-locals} and Bernstein's inequality in $x_3$ (see, for example, \cite[Ch.5]{Wolff}), we obtain
\begin{align}\label{j<n<k error 2} 
\big\| \chi_{\mathbb J_\tau} (t,s)   e^{\pm it|\bar D|}\tilde E_N^\pm(s|D|) f \big\|_{L^q_{{x},t}(\mathbb R^3\times \mathbb I)}  \lesssim     (\tau h)^{-N'}  h^{\frac{1}{p}-\frac{1}{q}} \lambda^{\frac12+\frac{1}{p}-\frac{3}{q}+\epsilon} 
\Vert f\Vert_{L^p} 
\end{align}
for $p,q$ satisfying  ${1}/{p}+{3}/{q}\le 1$ since $\supp \widehat f \subset \bA_\lambda \times \bI_h$, $s\in \bI_\tau$,  and $\tau h \gtrsim 1$. 
Recalling \eqref{bessel analog}, we consider the multiplier operator given by 
\[ \tx  a_{l,t,s}^\pm (\xi)=\sum_{\pm, 0\le j\le N} B_j^{\pm}|s{\xi}|^{-\frac{1}{2}-j}.\]
 Since $\lambda \lesssim  1/\tau\lesssim h$,  using the same argument as before   (e.g., Lemma \ref{fff}), we may replace 
 $e^{\pm i|s{\xi}|}$ with $e^{\pm i|s{\xi_3}|}$. 
 By the Mikhlin multiplier theorem,  we have 
\[
\big\| \chi_{\mathbb J_\tau} (t,s)  e^{\pm i(t|\bar D|+ s|D|)}  a_{l,t,s}^\pm (D) f\big\|_{L^q_{{x},t}(\mathbb R^3\times \mathbb I)}  \lesssim 
(\tau h)^{-\frac12}  \Big\| \chi_{\mathbb J_\tau} (t,s)  e^{\pm it|\bar D|} f \big \|_{L^q_{{x},t}(\mathbb R^3\times \mathbb I)}.
\]
Applying \eqref{locals} and Bernstein's inequality as before, we have the left hand side bounded by 
$
(\tau h)^{-\frac 12}  h^{\frac{1}{p}-\frac{1}{q}} \lambda^{\frac12+\frac{1}{p}-\frac{3}{q}+\epsilon} 
\Vert f\Vert_{L^p}
$
for ${1}/{p}+{3}/{q}\le 1$. 
Combining  this and  \eqref{j<n<k error 2}, we obtain 
\[
\big\|  \chi_{\mathbb J_\tau} (t,s)M_l^\pm(D,t,s) f
\big\|_{L^q_{{x},t}(\mathbb R^3\times \mathbb I)}\lesssim 
(\tau h)^{-\frac 12}  h^{\frac{1}{p}-\frac{1}{q}} \lambda^{\frac{1}{p}-\frac{3}{q}-l +\epsilon} 
\Vert f\Vert_{L^p}.
\]
Thus, taking integration in $s$ gives 
\eqref{tml}.

 We now consider the contribution of the error term $\mathcal E$ in  \eqref{exp}, whose  contribution  is less significant. 
 It can be handled by using the estimates for fixed $(t,s)\in \mathbb J_\tau$. Recalling \eqref{exp}, we  set 
\[  E_N^{\,0}(\theta):=E_N^{\,0}(\theta, s,t, \bxi)= |\bxi|^{N'}
 E_N((t+s\cos\theta)|\bxi|).\]
 We have  $|\partial_{\theta}^n  E_N^{\,0}(\theta)| \lesssim 1$ uniformly in $n$, $\theta$ for $(t,s)\in \mathbb J_\tau$ since $(t+s\cos\theta)\gtrsim 1-c_0$ for $(t,s)\in \mathbb J_\tau$. By the stationary phase method  \cite[Theorem 7.7.5]{hormander-book} one can obtain a similar expansion as before: 
\begin{equation}\label{bessel2}
\int e^{-is\xi_3\sin\theta} E_N^{\,0}(\theta) d\theta=\sum_{\pm, 0\le w\le M} D_w^{\pm}|s\xi_3|^{-\frac{1}{2}-w}e^{\pm is\xi_3}+E_M'(|s\xi_3|)
\end{equation}
for $(t,s)\in \mathbb J_\tau$. 
Here, $E_M'$ satisfies the same bounds as $E_N$ (i.e., \eqref{bessel error decay}) and $M\le N/4$.  $D_w^{\pm}$ and $E_M'$ depend on $t,\xi$, but they are harmless as can be seen  by the Mikhlin multiplier theorem. The contribution from $E_M'$ can be directly controlled by the Mikhlin multiplier theorem. Since $\supp f\subset \bA_\lambda \times \bI_h$,  Bernstein's inequality gives 
\[
\Big\| \int e^{-isD_3\sin\theta}E_N((t+s\cos\theta)|D|) d\theta f \Big\|_{L^q_{{x}}}\lesssim   (\tau h)^{-\frac12}  \lambda^{-N'} (\lambda^2 h)^{\frac{1}{p}-\frac{1}{q}}\Vert f\Vert_{L^p}
\]
for $(t,s)\in \mathbb J_\tau$.  
Note that the implicit constant here does not depend on $t,s$.  
Thus, integration in $s,t$ gives 
\Be
\label{error} 
 \| \mathcal E(D,t,s) f\|_{L^q(\R^3\times \mathbb J_\tau)} \le C \tau^\frac1q (\tau h)^{-\frac12}  h^{\frac{1}{p}-\frac{1}{q}} \lambda^{2-N'}\|f\|_p
 \Ee
 for $1\le p\le q\le \infty$. So,  the contribution of  
$\mathcal E(D,t,s) f$ is acceptable. Therefore, 
from \eqref{exp} and \eqref{tml}, we obtain \eqref{0jnk local}.

Putting \eqref{exp}, \eqref{ml}, \eqref{ce}, and \eqref{bessel analog} together, by Plancherel's theorem one can easily see 
$ \| \mathcal A_t^s f\|_{L_x^2}\lesssim (\tau h)^{-\frac12} \lambda^{-\frac12} \|f\|_2.$ Thus, integration in $s,t$ gives 
\Be 
\label{l2l2} 
\| \mathcal A_t^s f\|_{L^2(\R^3\times \mathbb J_\tau)}\lesssim h^{-\frac12} \lambda^{-\frac12} \|f\|_2, 
\Ee
which is 
\eqref{0jnk local1} for $p=q=2$. Interpolation between this  and the estimate  \eqref{0jnk local} for $p,q$ satisfying $1/p+3/q= 1$ gives   \eqref{0jnk local1} for $1/p+3/q> 1$. 
\end{proof}

\subsection{When $ \supp \widehat f \subset \mathbb A_\lambda^\circ\times \mathbb R$ and $\lambda\gtrsim 1/\tau$}
We have the following estimate.

\begin{prop}\label{A1 estimate}  Let $2\le p\le q\le \infty$ satisfy  $1/p+1/q\le 1$. 
$(a)$ If $1/\tau\lesssim \lambda\lesssim h\lesssim \tau\lambda^2$,  then for any $\epsilon>0$  we have  the estimates
\begin{align}\label{A1 ineq}
 \Vert \mathcal A_t^{s} f\Vert_{L^q(\R^3\times \mathbb J_\tau)}  &\lesssim   
 \tau^{\frac{3}{2q}-\frac{1}{2}-\frac{1}{2p}} h^{-\frac{1}{2}+\frac{3}{2p}-\frac{3}{2q}+\epsilon} \lambda^{ \frac{1}{2p}-\frac{1}{2q}-\frac{1}{2}}\Vert f\Vert_{L^p}  
\end{align}
 for $1/p+ 3/{q}> 1$, and 
\begin{align} 
\label{A1 ineq1}
\Vert \mathcal A_t^{s} f\Vert_{L^q(\R^3\times \mathbb J_\tau)}  &\lesssim    \tau^{-\frac1{p}}h^{-1+\frac{2}{p}+\epsilon} \lambda^{1-\frac{1}{p}-\frac{5}{q}}\Vert  f\Vert_{L^p} 
\end{align}
 for $1/{p}+3/{q}\le 1$ whenever $\supp \widehat f \subset\bA_\lambda \times \bI_h$.  $(b)$
If $\supp \widehat f \subset\bA_\lambda \times \bI_\lambda^\circ$,  the estimates \eqref{A1 ineq} and  \eqref{A1 ineq1} hold with $h=\lambda$.  $(c)$ Suppose $1/\tau\lesssim \lambda$ and $h\gtrsim \lambda^2 \tau$, then the estimates \eqref{0jnk local} and \eqref{0jnk local1} hold whenever $\supp \widehat f \subset\bA_\lambda \times \bI_h$.
\end{prop}

 We can prove Proposition \ref{A1 estimate}  in the same manner as  Proposition \ref{case20}, using the expansions \eqref{exp} and \eqref{bessel analog}.  By \eqref{error}  we may disregard the contribution from  $\mathcal E$. Thus, we only need to handle $M_l^\pm$.  Moreover,  one can easily see the contribution from the multiplier operator  $\tilde E_N^\pm (s|D|)$ is acceptable. 
In fact, we have the following. 
 
\begin{lem}
Let $2\le p\le q\le \infty$ and  $1/p+1/q\le 1$.  If  $\supp \widehat f \subset \bA_\lambda \times \bI_h$ and $h\gtrsim \lambda,$ then we have  the estimates 
\begin{align}\label{A^1 error} 
\big\| |\bar D|^{-\frac12} e^{\pm it|\bar D|}\tilde E_N^\pm (s|{D}|)
f \big\|_{L^{q}(\R^3\times \mathbb J_\tau)} &\lesssim
\tau^\frac1q (\tau h)^{-N'}  h^{\frac{1}{p}-\frac{1}{q}} \lambda^{ \frac{1}{p}-\frac{3}{q}+\epsilon}\Vert f\Vert_{L^p} 
\end{align}
for $1/{p}+3/{q}\le 1$, and 
\begin{align}
\label{A^1 error2} 
\big\| |\bar D|^{-\frac12} e^{\pm it|\bar D|}\tilde E_N^\pm (s|{D}|)
f \big\|_{L^{q}(\R^3\times \mathbb J_\tau)}&\lesssim
\tau^\frac1q (\tau h)^{-N'}  h^{\frac{1}{p}-\frac{1}{q}} \lambda^{\frac{3}{2p}-\frac{3}{2q}-\frac12+\epsilon}\Vert f\Vert_{L^p}
\end{align}
for $1/{p}+3/{q}> 1$. If $\supp \widehat f \subset \bA_\lambda \times \bI_\lambda^\circ$,  \eqref{A^1 error}  and \eqref{A^1 error2}  hold with $h=\lambda$.
\end{lem}

\begin{proof} 
We first consider the case $\supp \widehat f \subset \bA_\lambda \times \bI_h$ and $h\gtrsim \lambda.$
The estimate \eqref{A^1 error} is easy to show by using  \eqref{locals-est} and Bernstein's inequality (for example,  see \eqref{j<n<k error 2}). 
Note that \eqref{A^1 error2}  with $p=q=2$ follows by Plancherel's theorem. Thus, interpolation between this estimate and \eqref{A^1 error}  for $1/{p}+3/{q}=1$ gives \eqref{A^1 error2} for  $1/{p}+3/{q}>1$.
If $\supp \widehat f \subset \bA_\lambda \times \bI_\lambda^\circ$, the estimates \eqref{A^1 error}  and \eqref{A^1 error2} with $h=\lambda$ follow in the same manner. We omit the detail. 
\end{proof}

\begin{proof}[Proof of Proposition \ref{A1 estimate}]  Recalling  \eqref{bessel analog} and comparing the estimates \eqref{A^1 error} and \eqref{A1 ineq}, we notice  that  it is sufficient to consider  the estimates for the multiplier operators defined by 
$B_j^{\pm}|s{\xi}|^{-\frac{1}{2}-j}e^{\pm i|s{\xi}|}$. Therefore, the matter is reduced to obtaining, instead of  $\mathcal A_t^s$,  the estimates   
 for the operators 
\Be 
\label{central}
\mathcal{C}_\pm^\kappa  f(x,t,s)  :=|\bar D|^{-\frac12}|sD|^{-\frac12} \mathcal U f(x,\kappa t, \pm  s), \quad \kappa = \pm,
\Ee
which constitute the major part. 
We first consider the case $(a)$: $1/\tau\lesssim \lambda\lesssim h\lesssim \tau\lambda^2$ and $\supp \widehat f \subset\bA_\lambda \times \bI_h$. 
Note that $\| \mathcal{C}_\pm^\kappa f(\cdot, s,t)\|_{L^{q}(\R^3)}\lesssim (\tau\lambda h)^{-\frac12} \|\mathcal Uf(\cdot,\kappa t, \pm s)\|_{L^{q}(\R^3)} $ for $\kappa=\pm$. 
Thus, by \eqref{plocal1} and Remark \ref{pm} we get  
\begin{align*} 
 \|\mathcal{C}_\pm^{\kappa}f\|_{L^{q}(\R^3\times \mathbb J_\tau)} 
 \lesssim 
  \tau^{-\frac1{p}}h^{-1+\frac{2}{p}+\epsilon} \lambda^{1-\frac{1}{p}-\frac{5}{q}} \Vert  f\Vert_{L^p},  \quad \kappa=\pm \end{align*}
for $1/{p}+3/{q}\le 1$. Therefore,  we obtain 
\eqref{A1 ineq1}. So, \eqref{A1 ineq} follows from interpolation with \eqref{l2l2}.

If $\supp \widehat f \subset \bA_\lambda \times \bI_\lambda^\circ$, 
by  the estimate  \eqref{plocal1}  with $\lambda=h$ ($(b)$ in Lemma \ref{flocal}) we get the desired estimates \eqref{A1 ineq1} and \eqref{A1 ineq}  with $h=\lambda$. 
This proves $(b)$.

If $1/\tau\lesssim \lambda$, $h\gtrsim \lambda^2 \tau$, and $\supp \widehat f \subset\bA_\lambda \times \bI_h$, 
the estimate \eqref{0jnk local} follows by \eqref{plocal2}.  As a result,  we get \eqref{0jnk local1} by interpolation between \eqref{l2l2} and \eqref{0jnk local}.
\end{proof}

Since  the main contribution to the estimate for $\cA_t^s f$ is from $\mathcal C_t^s f$, by  the same argument in the proof of Proposition \ref{A1 estimate} 
one can easily obtain the next.   

\begin{cor} Let $\alpha, \beta\in \mathbb N_0$. 
\label{cor:A1 estimate} 
$(a)$ If $1/\tau\lesssim \lambda\lesssim h\lesssim \tau\lambda^2$,  then for any $\epsilon>0$  
\begin{align*}
 \Vert \partial_t^\alpha\partial_s^\beta \mathcal A_t^{s} f\Vert_{L^q(\R^3\times \mathbb J_\tau)}  &\lesssim   
 \tau^{\frac{3}{2q}-\frac{1}{2}-\frac{1}{2p}} h^{\beta-\frac{1}{2}+\frac{3}{2p}-\frac{3}{2q}+\epsilon} \lambda^{ \alpha+ \frac{1}{2p}-\frac{1}{2q}-\frac{1}{2}}\Vert f\Vert_{L^p}, &&1/p+ 3/{q}> 1, 
\\
\Vert \partial_t^\alpha\partial_s^\beta \mathcal A_t^{s} f\Vert_{L^q(\R^3\times \mathbb J_\tau)}  &\lesssim    \tau^{-\frac1{p}}h^{\beta-1+\frac{2}{p}+\epsilon} \lambda^{\alpha+ 1-\frac{1}{p}-\frac{5}{q}}\Vert  f\Vert_{L^p},  &&1/{p}+3/{q}\le 1, 
\end{align*}
 hold  whenever $\supp \widehat f \subset\bA_\lambda \times \bI_h$.  $(b)$
If $\supp \widehat f \subset\bA_\lambda \times \bI_\lambda^\circ$, we obtain the above two estimates 
with $h=\lambda$.  $(c)$ When $1/\tau\lesssim \lambda$ and $h\gtrsim \lambda^2 \tau$,
 for any $\epsilon>0$ we have
\begin{align*}
\Vert \partial_t^\alpha\partial_s^\beta\mathcal A_t^{s}  f\Vert_{L^q(\R^3\times \mathbb J_\tau)}&\lesssim
\tau^\frac1q (\tau h)^{-\frac12}  h^{\beta+ \frac{1}{p}-\frac{1}{q}} \lambda^{\alpha+ \frac{1}{p}-\frac{3}{q}+\epsilon}\Vert f\Vert_{L^p},  \,  &&1/{p}+3/{q}\le 1,
\\
\Vert  \partial_t^\alpha\partial_s^\beta\mathcal A_t^{s}  f\Vert_{L^q(\R^3\times \mathbb J_\tau)}&\lesssim
\tau^\frac1q (\tau h)^{-\frac12}  h^{\beta+ \frac{1}{p}-\frac{1}{q}} \lambda^{\alpha-\frac12+  \frac{3}{2p}-\frac{3}{2q}+\epsilon}\Vert f\Vert_{L^p},   \, && 1/{p}+3/{q}> 1, 
\end{align*}
 whenever $\supp \widehat f \subset\bA_\lambda \times \bI_h$.
\end{cor}

\begin{rmk}
By \eqref{exp} and \eqref{bessel analog}  it follows that 
 \[ |\widehat{d\sigma_t^{s}}({\xi})|\lesssim (1+|\xi_3|)^{-1/{2}}(1+|{\bxi}|)^{-1/{2}}.\] 
Furthermore, if $|\bxi|\lesssim 1$, we have $|\widehat{d\sigma_t^{s}}({\xi})|\sim |{\xi}|^{-1/{2}}$ for $|\xi|$ large enough.  
Therefore, by Plancherel's theorem one can see that  the $L^2$ to $L^2_{1/2}$ estimate for $\cA_t^s$  is  optimal.  One can also see that  the part of the  surface $\mathbb T_t^s$ near the sets $\{ \Phi_s^t(\pm \pi/2, \phi): \phi\in [0, 2\pi)\}$ is responsible for the worst decay  
while the Fourier transform of the part (of the surface) away from the sets enjoys better decay. 
\end{rmk}

\section{Two-parameter maximal and smoothing estimates}\label{sec3}
In this section we prove Theorem \ref{global maximal},  \ref{main thm}, and \ref{two para smoothing}.
First, we recall an elementary lemma, which enables us to relate the local smoothing estimate to the estimate for the maximal function. 
\begin{lem}\label{sobo} Let $1\le p\leq \infty$, and  let $I$ and $J$ be closed intervals of length $1$ and $\ell$, respectively. 
Suppose $G$ be a smooth function on the rectangle $R=I\times J$. 
Then, for any $\lambda, h>0$, we have
\begin{align*}
 \sup_{(t,s)\in I\times J}\vert G(t,s) \vert  &\lesssim (1+\lambda^{1/p})(\ell^{-1/p}+ h^{1/p}) \Vert G\Vert_{L^p(R)}+ (\ell^{-1/p}+ h^{1/p})  \lambda^{-1/p'}  \Vert\partial_tG\Vert_{L^p( R)} 
 \\
 &+  
 (1+ \lambda^{1/p})  h^{-1/p'}  \Vert\partial_s G\Vert_{L^p( R)}  + \lambda^{-1/p'}h^{-1/p'}  \Vert\partial_t\partial_s G\Vert_{L^p( R)}. 
 \end{align*}
\end{lem}

\begin{proof}
We first recall the inequality 
\[\tx \sup_{t\in I'}\vert F(t) \vert  \lesssim    {|I'|^{-1/p}}\Vert F\Vert_{L^p( I')}+\Vert F\Vert_{L^p( I')}^{{(p-1)}/{p}}\Vert\partial_tF\Vert_{L^p( I')}^{1/{p}},\] which holds 
whenever $F$ is a smooth function defined on an interval  $I'$ (for example, see \cite{L}). 
By Young's inequality we have 
\[  \tx\sup_{t\in I'}\vert F(t) \vert  \lesssim   {|I'|^{-1}}\Vert F\Vert_{L^p( I')}+\lambda^{1/p} \Vert F\Vert_{L^p( I')}+  \lambda^{-1/p'}  \Vert\partial_tF\Vert_{L^p( I')}. \]
for any $\lambda>0$.  We use this inequality with $F=G(\cdot, s)$ and $I'=I$  to get    
\[  \sup_{(t,s)\in I\times J}\vert G(t,s) \vert  \lesssim (1+\lambda^\frac1p)  \Vert \sup_{s\in J} |G(t,s)| \Vert_{L^p(I)}  + 
 \lambda^{-1/{p'}}  \Vert \sup_{s\in J} |\partial_t G(t,s)|\Vert_{L^p( I)}. \]
Then, we apply the above inequality again to $G(t,\cdot)$ and $\partial_t G(t,s)$ with $I'=J$ taking $\lambda=h$. 
\end{proof}

In what follows, we frequently use the Littlewood-Paley decomposition.  Let  $\varphi\in C_c^\infty((1-2^{-13},2+2^{-13}))$ such that $\sum_{j=-\infty}^\infty \varphi(s/2^j)=1$ for $s>0$. We set $\varphi_j(s)=\varphi(s/2^j)$, $\varphi_{<j}(s)=\sum_{k<j}\varphi_k(s)$, and $\varphi_{>j}(s)=\sum_{k>j}\varphi_k(s)$. 
For a given $f$ we define  $f^k_j$ and $f^{ <k}_{<j}$ by 
\begin{align*}
\mathcal F (f^k_j) =\varphi_j(|\bar\xi|) \varphi_k(|\xi_3|) \widehat f(\xi),
\quad \mathcal F (f^{ <k}_{<j} ) =\varphi_{{<j}} (|\bar\xi|) \varphi_{<k}(|\xi_3|) \widehat f(\xi),
\end{align*}
and  $f^{ <k}_{<j} $, $f^{ k}_{{<j}} $,    $f^{ \ge k}_{{j}} $, $f_{{<j}} $, and $f^{\ge k} $, etc are similarly defined. In particular, we have $f=\sum_{j,k} f_j^k$.

\subsection{Proof of Theorem \ref{global maximal}}   
By a standard  argument using  scaling,  it is sufficient to show $L^p$ boundedness of  a localized  maximal operator 
\[ \mathfrak{M}f({x})=\sup_{0<s<c_0 t<1}\big| \mathcal A_{t}^sf({x}) \big|. \]
Furthermore, we only need to show that $\mathfrak M$ is bounded on $L^p$ for $2<p\le 4$ since the
other estimates follow by interpolation with the trivial $L^\infty$ bound.  
To this end,   we  consider  
\begin{equation}
\label{Mn}
\mathfrak{M}_{n}f({x})=\sup_{(t,s)\in \mathbb J_{2^{-n}}}\big|\mathcal A_t^{s}f({x})\big|, \quad n \ge 0. 
\end{equation}
In order to obtain estimates for $\mathfrak M_n$, we consider $\mathfrak M_n f_j^k$ for each $j,k$. 
The correct bounds in terms of  $n$, not to mention $j, k$,   are also important for our purpose.  

\begin{lem}
\label{flmax}
Let $k, j\ge n$. $(\tilde a)$ If $ j\leq k\leq 2j-n$, we have
\begin{align}\label{al<j<k<2j-al}
 \Vert \mathfrak M_n f_j^k \Vert_{L^q} \lesssim \begin{cases}
2^{n(\frac{1}{2}+\frac{1}{2p}-\frac{3}{2q})+j(\frac{1}{2p}+\frac{1}{2q}-\frac{1}{2})+k(\frac{3}{2p}-\frac{1}{2q}-\frac{1}{2}+\epsilon)}\Vert f\Vert_{L^p}, & \frac{1}{p}+\frac{3}{q}\geq 1,\\
2^{\frac{n}{p}+j(1-\frac{1}{p}-\frac{4}{q})+k(\frac{2}{p}+\frac{1}{q}-1+\epsilon)}\Vert f\Vert_{L^p},  & \frac{1}{p}+\frac{3}{q}< 1.
\end{cases}
\end{align}
$(\tilde b)$ For $\mathfrak M_n f_j^{<j}$, the same bounds hold with $k=j$. 
$(\tilde c)$ 
If  $2j-n\leq k$, then we have
\begin{equation}\label{al<j,2j-al<k}
\Vert \mathfrak M_n f_j^k \Vert_{L^q}\lesssim \begin{cases}
2^{n(\frac{1}{2}-\frac{1}{q})+j(\frac{3}{2p}-\frac{1}{2q}-\frac{1}{2}+\epsilon)+k(\frac{1}{p}-\frac{1}{2})}\Vert f\Vert_{L^p},\quad & \frac{1}{p}+\frac{3}{q}\geq 1,\\
2^{n(\frac{1}{2}-\frac{1}{q})+j(\frac{1}{p}-\frac{2}{q}+\epsilon)+k(\frac{1}{p}-\frac{1}{2})}\Vert f\Vert_{L^p}, \quad & \frac{1}{p}+\frac{3}{q}< 1.
\end{cases}
\end{equation}
\end{lem} 

\begin{proof}  
 Let  $n_0$ be the smallest integer such $ 2^{-n_0+1}\le c_0$. 
If $n\ge n_0$, then $\mathbb J_{2^{-n}}= \mathbb I\times \mathbb I_{2^{-n}}$. 
Since $n\le k, j$, using Lemma \ref{sobo}, one can obtain $(\tilde a)$, $(\tilde b)$, and $(\tilde c)$ from  $(a)$, $(b)$, and $(c)$ in Corollary \ref{cor:A1 estimate}, respectively. 
For $n< n_0$,  we can not directly apply Lemma \ref{sobo}. However, this can be  easily overcome by a simple  modification. Indeed,  we cover $\bigcup_{n=0}^{n_0-1} \mathbb J_{2^{-n}}$ with essentially 
disjoint closed dyadic cubes $Q$ of side length $L\in ( 2^{-7} (1-c_0), 2^{-6} (1-c_0)]$  so that  $\bigcup Q\subset \mathbb J_0':= \{(t,s): 2^{1-n_0} \le s< 2^{-1}(1+c_0)t, 1\le t\le 2\}$.
Thus, we note 
\[ \tx \| \sup_{(t,s)\in \mathbb J_{2^{-n}}} |\mathcal A_t^{s}g |\big\|_{L^q} \lesssim \sum_Q \| \sup_{(t,s)\in Q} |\mathcal A_t^{s}g |\big\|_{L^q}. \] 
for  $n< n_0$.  We  may now apply Lemma \ref{sobo} to  $\mathcal A_t^{s}g$ and $Q$. Since $\bigcup Q\subset \mathbb J_0'$, we  clearly have the same maximal bounds  up to a constant multiple 
for $n< n_0$.  
\end{proof}

 We denote $\mathrm Q_l^m=\mathbb J_0\cap (\bI_{2^{-l}}\times \bI_{2^{-m}})$ for simplicity. Then, it follows that   
\[  \mathfrak{M}f(x)  = \sup_{m\geq l\geq 0}\sup_{(t,s)\in  \mathrm Q_l^m}|\mathcal A_t^{s}f|.\]
Decomposing $f=\sum_{j,k} f_j^k$, we have 
\begin{align*}
\mathfrak{M}f(x) \le  \mathfrak N^1f + \mathfrak  N^2f +\mathfrak N^3f+\mathfrak N^4f,
\end{align*}
where
\begin{align*}
 \mathfrak N^1f =  \sup_{m\geq l\geq 0}\sup_{(t,s)\in  \mathrm Q_l^m}|\mathcal A_t^{s}f_{\le l}^{\le m}|, 
 \qquad 
  \mathfrak N^2 f =  \sup_{m\geq l\geq 0}\sup_{(t,s)\in  \mathrm Q_l^m}|\mathcal A_t^{s}f_{\le l}^{> m} |,
 \\
 \mathfrak N^3 f =  \sup_{m\geq l\geq 0}\sup_{(t,s)\in  \mathrm Q_l^m}|\mathcal A_t^{s} f_{> l}^{\le m} |, 
 \qquad  
 \mathfrak N^4 f =  \sup_{m\geq l\geq 0}\sup_{(t,s)\in  \mathrm Q_l^m}|\mathcal A_t^{s} f_{> l}^{> m} |.
\end{align*}
The maximal operators 
$\mathfrak N^1,\mathfrak N^2$ and $\mathfrak N^3$ can be handled by using the $L^p$ bounds on the Hardy-Littlewood maximal and the circular maximal functions.

We first handle $\mathfrak N^1f$.  We set 
$\bar K =\mathcal F^{-1}(\varphi_{\le 1 }(|\,\bxi\,|))$ and $K_3=  \mathcal F^{-1}(\varphi_{\le 1 }(|\xi_3|))$.  
Since $\mathcal F(f_{\le l}^{\le m})(\xi) =  \varphi_{\le l }( \bxi) \varphi_{\le m} (\xi_3)  \widehat f (\xi)$ and $\varphi_{\le m }(t)= \varphi_{\le 1 }(2^{-m}t)$, we have
\[ f_{\le l}^{\le m}(x)=   2^{2l+m} \int f({x}-{y})\bar K(2^l \bar y)K_3 (2^m y_3) dy .\]  
Hence, it follows that 
\[ \mathcal A_t^{s} f_{\le l}^{\le m} ({x})=2^{2l+m} \int_{\mathbb{T}_t^{s}}   \int f({x}-{y})\bar K(2^l (\bar y-\bar z) )K_3 (2^m(y_3-z_3)) dy\, d\sigma_{t}^s (z).\]
If $(t,s)\in  \mathrm Q_l^m$,  $|\bar K(2^l (\bar y-\bar z) )K_3 (2^m (y_3-z_3)|\le C(1+2^l|\bar y|)^{-M}(1+2^m|y_3|)^{-M}$ for any $M$. 
By a standard argument using dyadic decomposition, we see  
\[ \mathfrak N^1f(x) \lesssim   \bar HH_3 f(x), \]
where $\bar H$ and $H_3$ denote the $2$-d and $1$-d Hardy-Littlewood maximal operators acting on $\bar x$ and $x_3$, respectively.  The right hand side is bounded by the strong maximal function. Thus,   $\mathfrak N^1$ is bounded on $L^p$ whenever $p>1$.
 
  Next, we consider $\mathfrak N^2$. Since $f_{\le l}^{> m}(x)=2^{2l} \big(f^{>m}(\cdot, x_3)\ast \bar K(2^l\cdot)\big)(\bar x)$, we have   \begin{align*}
\mathcal A_t^{s}f_{\le l}^{> m} =2^{2l}\int f^{>m}(\bar x-\bar y,x_3-s\sin\theta) \bar K(2^l(\bar y-  (t+s\cos\theta)\vphi )) d\theta d\phi d\bar y.
\end{align*}
Note that $s<c_0t\lesssim 2^{-l}$, so we have $|\bar K(2^l(\bar y-  (t+s\cos\theta)\vphi ))|\lesssim  C(1+2^l|\bar y|)^{-M}$ for any $M$. 
Similarly as above, this gives 
\[ 
|\mathcal A_t^{s}f_{\le l}^{> m}(x) |\lesssim  \int_0^{2\pi} \bar H f^{> m}( \bar x, x_3-s\sin\theta) d\theta  \lesssim  \int_0^{2\pi} \bar H H_3 f( \bar x, x_3-s\sin\theta) d\theta
\]
 For the second inequality, we use $f^{>m}=f-f^{\le m}$ and $|f|, |f^{\le m}|\le H_3 f$.  As a result, we have
\[  \mathfrak N^2 f(x)\lesssim \sup_{s>0} \int_0^{2\pi} \bar H H_3 f( \bar x, x_3-s\sin\theta) d\theta.  \] 
To handle the consequent maximal operator, we use the following simple lemma.
\begin{lem}\label{reduced circular}  
For $p>2$, we have the estimate 
\[ \Big\| \sup_{0<s<1}\Big| \int g(x_3-s\sin\theta)d\theta \Big| \Big\|_{L^p_{x_3}} \lesssim \Vert g\Vert_{L^p}.\]
\end{lem}
\begin{proof}
Let us define $\widetilde g$ on $\mathbb R^2$ by setting $\widetilde{g}(z, x_3)=g(x_3)$ for $x_3\in\R$ and $-10\leq z\leq 10$, and $\widetilde{g}(z,x_3)=0$ if $|z|>10$. 
Note that $\int g(x_3-s\cos\theta)d\theta= \int \widetilde{g}(z-s\cos\theta,x_3-s\sin\theta)d\theta$ for $|z|\le 1, 0< s<1$. So,  
$\sup_{0<s<1}| \int g(x_3-s\sin\theta)d\theta| \lesssim M_{cr} \widetilde g(z,x_3)$ for $|z|\le 1$, where $M_{cr}$ denotes the circular maximal operator.
By the circular maximal theorem \cite{B2}, $\| \sup_{0<s<1} | \int g(x_3-s\sin\theta)d\theta \|_{L^p_{x_3}} $ is bounded above by a constant times 
$\Vert \widetilde{g}\Vert_{L^p_{x_3,z}}=20^{1/p}\Vert g\Vert_{L^p_{x_3}}$
for $p>2$.
\end{proof}
 
 Therefore, by Lemma \ref{reduced circular} and $L^p$ boundedness of $\bar H$ and $H_3$  we see that $\mathfrak N^2$ is bounded on $L^p$ for $p>2$.

 $\mathfrak N^3$ can be handled similarly. Since  $f_{>l}^{\le m}=2^{m} \big(f_{>l}(\bx, \cdot)\ast K_3(2^m\cdot)\big)(x_3)$, we get 
  \begin{align*}
\mathcal A_t^{s} f_{> l}^{\le m}({x})=2^m\int f_{>l}(\bx-(t+s\cos\theta)\vphi,x_3-y_3) K_3(2^m(y_3- s\sin\theta) ) d\theta d\phi dy_3.
\end{align*}
Since  $s\lesssim 2^{-m}$,  $| K_3(2^m(y_3- s\sin\theta) )|\lesssim (1+ 2^m|y_3|)^{-N}$. Hence,  using $f_{>l}=f-f_{\le l}$ and $|f|, |f_{\le l}|\le\bar H f$, we have 
\[ |\mathcal A_t^{s} f_{> l}^{\le m}({x})|\lesssim  \int_{0}^{2\pi} H_3\bar Hf(\bx-(t+s\cos\theta)\vphi,x_3)d\phi  \lesssim   M_{cr} [(H_3\bar Hf)(\cdot, x_3)](\bar x).   \]  Thus,  $\mathfrak N_3f(x) \lesssim  M_{cr} [(H_3\bar Hf)(\cdot, x_3)](\bar x)$. Using the circular maximal theorem, we see that  $\mathfrak N^3$ is bounded on $L^p$ for $p>2$.

Finally, we consider $\mathfrak N^4$.   For simplicity,  
we set 
\[ \tx\mathfrak A_{l,j}^{m,k} f =\sup_{(t,s)\in  \mathrm Q_l^m}|\mathcal A_t^{s}f_j^k|.  \]
Decomposing $\sum_{j\geq l,k\geq m} = \sum_{m\le k \le j} + \sum_{j< k\le 2j-m} +  \sum_{l\le j,\, m\vee(2j-m)< k }$,  we have
\begin{align*}
\mathfrak N^4f \le  \sup_{m\geq l\geq 0} \mathfrak S_1^{m,l} f +   \sup_{m\geq l\geq 0} \mathfrak S_2^{m,l}  f + \sup_{m\geq l\geq 0} \mathfrak S_3^{m,l} f, 
 \end{align*}
 where 
  \[ \mathfrak S_1^{m,l} f =\sum_{m\le k \le j} \mathfrak A_{l,j}^{m,k} f , \qquad   \mathfrak S_2^{m,l} f =   \!\!\! \sum_{j< k\le 2j-m}  \!\!\!  \mathfrak A_{l,j}^{m,k} f, \qquad   
  \mathfrak S_3^{m,l} f =  \!\!\! \sum_{l\le j,\, m\vee(2j-m)< k }  \!\!\! \mathfrak A_{l,j}^{m,k} f. \] 
  Here, $a\vee b$ denotes $\max(a,b)$. 
 Thus, the matter is reduced to showing, for $\kappa=1,2,3$,  
\Be\label{est123} \big\|  \sup_{m\geq l\geq 0} \mathfrak S_\kappa^{m,l} f \big\|_{L^p}  \lesssim  C\|f\|_p, \quad   p\in (2,4].\Ee

 We consider $\mathfrak S_1^{m,l}$ first. 
 Recalling \eqref{Mn},  by scaling we have  
 \Be 
 \label{jk}  \mathfrak A_{l,j}^{m,k} f (x)
  =\mathfrak{M}_{m-l}(f_j^k(2^{-l}\cdot))(2^{l}{x})=\mathfrak{M}_{m-l}[f(2^{-l}\cdot)]_{j-l}^{k-l}(2^{l}{x}).
 \Ee
 So, reindexing $k\to k+l$ and $j\to j+l$ gives
  \begin{align*}
\tx \mathfrak S_1^{m,l} f (x) \leq  \sum_{m-l\leq k\leq j}\mathfrak{M}_{m-l} [f(2^{-l}\cdot)]_j^k(2^{l}{x}).
\end{align*}
Thus, the imbedding $\ell^p\subset \ell^\infty$ and  Minkowski's inequality yield 
\[\|  \sup_{m\geq l\geq 0} \mathfrak S_1^{m,l} f \|_{L^p}^p  \leq   \sum_{m\geq l\geq 0} \Big( \sum_{m-l\leq k\leq j}\big\Vert \mathfrak{M}_{m-l}[f(2^{-l}\,\cdot)]_j^k(2^{l}{\,\cdot})\big\Vert_{L^p}\Big) ^p.\]
We now use  $(\tilde b)$ in Lemma \ref{flmax} (with $n=m-l$)  for $ \mathfrak{M}_{m-l}[f(2^{-l}\,\cdot)]_j^k(2^{l}{\,\cdot})$.
 Thus, by the first estimate in  \eqref{al<j<k<2j-al} with $k=j$, we have
\[
\|  \sup_{m\geq l\geq 0} \mathfrak S_1^{m,l} f \|_{L^p}^p \lesssim  \sum_{m\geq l\geq 0} 2^{(m-l)p(\frac{1}{2}-\frac1p)}\Big(  \sum_{m-l\leq j} 2^{-2j(\frac12-\frac{1}{p})}2^{\epsilon j}\Vert f_{j+l}\Vert_{L^p} \Big )^p
\]
for any $\epsilon>0$ for  $2< p\leq 4$. Taking $\epsilon>0$ small enough, we have 
\[ \|  \sup_{m\geq l\geq 0} \mathfrak S_1^{m,l} f \|_{L^p}^p  \lesssim  \sum_{m\geq l\geq 0 } \sum_{m-l\leq j} 2^{-a(m-l)}2^{-bj} \Vert f_{j+l}\Vert_{L^p}^p   \]
for some positive numbers $a,b$ for $2< p\leq 4$. Changing the order of summation, we see the right hand side is bounded above by 
$C\sum_{j\ge 0}^\infty  2^{-bj} \sum_{l\ge 0}   \Vert f_{j+l}\Vert_{L^p}^p$, which is bounded by $C\|f\|_p^p$, as can be seen, for example,   using   the Littlewood-Paley inequality. 
Consequently, we obtain \eqref{est123} for $\kappa=1$.

We now consider $\mathfrak S_2^{m,l}$.  As before, by  the imbedding $\ell^p\subset \ell^\infty$,   Minkowski's inequality,  \eqref{jk}, and  reindexing $k\to k+l$ and $j\to j+l$,  we get
\[\big\|  \sup_{m\geq l\geq 0} \mathfrak S_2^{m,l} f \big\|_{L^p}^p  \leq   \sum_{m\geq l\geq 0} \Big( \sum_{j< k\le 2j-(m-l)}
\big\Vert \mathfrak{M}_{m-l}[f(2^{-l}\,\cdot)]_j^k(2^{l}{\,\cdot})\big\Vert_{L^p}\Big) ^p.\]
The first inequality in \eqref{al<j<k<2j-al} with $n=m-l$ gives
\[\big\|  \sup_{m\geq l\geq 0} \mathfrak S_2^{m,l} f \big\|_{L^p}^p  \leq   \sum_{m\geq l\geq 0}  2^{(m-l)p(\frac{1}{2}-\frac1p)} \Big( \sum_{j< k\le 2j-(m-l)}
2^{-(j+k)(\frac{1}{2}-\frac{1}{p})}2^{\epsilon k} \Vert f_{j+l}\Vert_{L^p} \Big) ^p\]
for any $\epsilon>0$ for $2<p\le 4$. Note that $m-l< j$ for the inner sum, which is bounded by a constant times $ \sum_{m-l\leq j}   2^{-2j(1/{2}-1/{p})}2^{\epsilon j} \Vert f_{j+l}\Vert_{L^p} $ by taking sum over $k$ with an $\epsilon>0$ small enough. 
Since $p>2$, similarly,  we have  
\[ \|  \sup_{m\geq l\geq 0} \mathfrak S_2^{m,l} f \|_{L^p}^p  \lesssim  \sum_{m\geq l\geq 0 } \sum_{m-l\leq j} 2^{-a(m-l)}2^{-bj} \Vert f_{j+l}\Vert_{L^p}^p   \]
for some $a,b>0$ and $2< p\leq 4$.  Thus, the right hand is bounded above by $C\|f\Vert_{L^p}^p$.  
This proves \eqref{est123} for $\kappa=2$.

Finally,  we consider  $\mathfrak S_3^{m,l} f$, which we can handle in the same manner as before. Via the imbedding $\ell^p\subset \ell^\infty$, \eqref{jk}, and reindexing after applying Minkowski's inequality  
we have 
\[\|  \sup_{m\geq l\geq 0} \mathfrak S_2^{m,l} f \|_{L^p}^p  \lesssim    \sum_{m\geq l\geq 0} \Big( \sum_{0\le j, n\vee (2j-n)< k } 
\big\Vert \mathfrak{M}_{n}[f(2^{-l}\,\cdot)]_j^k(2^{l}{\,\cdot})\big\Vert_{L^p}\Big) ^p,\]
where  $n:=m-l$. Breaking $ \sum_{0\le j, n\vee (2j-n)< k } =  \sum_{0\le j\le n\le k} +  \sum_{n<j, (2j-n)< k } $, we apply the first estimate in \eqref{al<j,2j-al<k} 
to get 
\[  \|\sup_{m\geq l\geq 0} \mathfrak S_2^{m,l} f \|_{L^p}^p  \lesssim     \sum_{m\geq l\geq 0}   2^{np(\frac{1}{2}-\frac{1}{p})}(  \mathrm S_1^p+ \mathrm S_2^p )  \]
for any $\epsilon >0$ and $2<p\le 4$, 
where 
\[  \mathrm S_1:=  \sum_{0\le j\le n\le k}  2^{(j+k)(\frac{1}{p}-\frac{1}{2})}2^{\epsilon j}\Vert f_{j+l}^{k+l} \Vert_{L^p}, \quad  \ \mathrm S_2:= \!\!  \sum_{n<j, (2j-n)< k } \!\!\! 2^{(j+k)(\frac{1}{p}-\frac{1}{2})}2^{\epsilon j}\Vert f_{j+l}^{k+l} \Vert_{L^p}. \]
For the second sum $\mathrm S_2$,  we note that $k> j>n$. Thus, taking $\epsilon>0$ small enough,  we get 
 \[ \sum_{m\geq l\geq 0}  2^{np(\frac{1}{2}-\frac{1}{p})} \mathrm S_2^p
 \lesssim  \sum_{m\geq l\geq 0}  \sum_{m-l\leq j} 2^{-a(m-l)}2^{-bj} \Vert f_{j+l}\Vert_{L^p}^p\]
for some $a,b>0$ since $p>2$. Thus, the right hand side is bounded by $C\| f \|_{L^p}^p$. 
To handle  $\mathrm S_1$,   
note that $ (\sum_{0\le j\le n\le k} 2^{(j+k)(\frac{1}{p} -\frac{1}{2})})^{p/p'}
\lesssim 2^{ n(p-1)(\frac{1}{p} -\frac{1}{2})}$. Thus, by H\"older's inequality we have
\[  \mathrm S_1^p
 \lesssim   2^{n(p-1)(\frac{1}{p}-\frac{1}{2})}  \sum_{0\le j\le  n\le k}   2^{(j+k)(-\frac{1}{2}+\frac{1}{p})}2^{\epsilon p j}  \Vert f_{j+l}^{k+l}\Vert_{L^p}^p.\]
Hence, changing the order of summation, we get 
\[ \sum_{m\geq l\geq 0}   2^{np(\frac{1}{2}-\frac{1}{p})} \mathrm S_1^p \lesssim   
\sum_{0\le j}   2^{j(\frac{1}{p}-\frac{1}{2}+ \epsilon p)} \mathrm S_{1,j}^p, \]
 where 
\[  \mathrm S_{1,j}^p=  \sum_{m\geq l\geq 0} \,\, \sum_{m-l\le k}   2^{(m-l)(\frac{1}{2}-\frac{1}{p})} 2^{k(-\frac{1}{2}+\frac{1}{p})}\Vert f_{j+l}^{k+l}\Vert_{L^p}^p.\]
Therefore, since  $2<p\le 4$, taking a sufficiently small $\epsilon>0$,  we obtain  the desired inequality $\sum_{m\geq l\geq 0}   2^{np(\frac{1}{2}-\frac{1}{p})} \mathrm S_1^p\lesssim  \|f\Vert_{L^p}^p$ if we show that $\mathrm S_{1,j}^p \lesssim  \|f\Vert_{L^p}^p$ for $0\le j$.  To this end, rearranging the sums, we observe 
\[    \mathrm S_{1,j}^p=  \sum_{0\le k}\,\, \sum_{0\le l}\,\, \sum_{l\le m\le  l+k}  2^{(m-l)(\frac{1}{2}-\frac{1}{p})} 2^{k(-\frac{1}{2}+\frac{1}{p})}\Vert f_{j+l}^{k+l}\Vert_{L^p}^p \lesssim  \sum_{0\le k} \,\,\sum_{0\le l}   \Vert f_{j+l}^{k+l}\Vert_{L^p}^p . \] 
Since $\sum_{0\le k}   \Vert f_{j+l}^{k+l}\Vert_{L^p}^p\lesssim   \Vert f_{j+l}\Vert_{L^p}^p$, by the same  argument as above  it follows that $ \mathrm S_{1,j}^p \le C\|f\Vert_{L^p}^p$. 
Consequently, we obtain \eqref{est123} for $\kappa=3$. 
\qed

\subsection{Proof of Theorem \ref{main thm}}
Since $\mathbb J$ is a compact subset of $\mathbb J_\ast$,  there are constants $c_0\in (0,1)$, and   $m_1, m_2>0$  such that 
\[ \mathbb J\subset \{(t,s): m_1 \le  s\le m_2, s<c_0t\}.\]  
Therefore, via finite decomposition and scaling it is sufficient  to show that the maximal operator 
\[   \mathfrak{M}_{c} f(x):= \sup_{(t,s)\in \mathbb J_0} |\cA_t^s f(x)|\]  is bounded from $L^p$ to $L^q$ for $(1/p,1/q)\in \operatorname{int} \mathcal Q$.  To do this,  we decompose $f=f_{\ge 0} +f_{<0}^{\ge 0}+ f_{<0}^{<0}$ to have
\[
\mathfrak{M}_{c} f\lesssim \mathfrak{M}_{c}  f_{\ge 0}  +\mathfrak{M}_{c}f_{<0}^{\ge 0}+\mathfrak{M}_{c} f_{<0}^{<0}.
\]
The last two operators are easy to deal with. As before, we have 
$\mathfrak{M}_{c} f_{<0}^{<0}(x)\lesssim 
(1+|\cdot|)^{-M} \ast |f|(x)$, hence $ \Vert \mathfrak{M}_{c} f_{<0}^{<0}\Vert_{L^q}\lesssim \Vert f\Vert_{L^p}$ for $1\le p\le q\le \infty$. 
Concerning  $\mathfrak{M}_{c}f_{<0}^{\ge 0}$, we use  Lemma \ref{sobo} and  \eqref{000st} to get 
\[  \|\mathfrak{M}_{c}f_{<0}^{k}\|_{L^q}\lesssim  2^{k(-\frac{1}{2}+\frac{1}{p})}\Vert f\Vert_{L^p},   \quad  1\le p\le q\le \infty, \] 
for $k\ge 0$.  So, it follows that $\|\mathfrak{M}_{c}f_{<0}^{\ge 0}\|_{L^q}\lesssim  \Vert f\Vert_{L^p}$ for $2< p\le q$. Thus, we only need to show that 
 $\mathfrak{M}_{c}  f_{\ge 0}$ is 
bounded from $L^p$ to $L^q$ for $(1/p,1/q)\in \operatorname{int} \mathcal Q$. 

Decomposing $f_{\ge 0}= \sum_{j\ge 0} (  f_j^{\mathsmaller{<j}} +  \sum_{j\leq k\leq 2j} f_j^k + \sum_{k> 2j} f_j^k)$, we have
\[\tx \mathfrak{M}_{c}  f_{\ge 0}\le    \sum_{j\geq 0} ( \mathfrak S_j^1 f+\mathfrak S_j^2 f),\]
where 
\[ \tx \mathfrak S_j^1 f=  \mathfrak{M}_{c}f_j^{\mathsmaller{<j}} +  \sum_{j\leq k\leq 2j}\mathfrak{M}_{c}  f_j^k, \qquad  \mathfrak S_j^2 f=\sum_{k> 2j}\mathfrak{M}_{c}f_j^k. \]

We first show $L^p$--$L^q$ bound on $\mathfrak{M}_{c}  f_{\ge 0}$ for $(1/p, 1/q)$ contained in the interior of the  triangle $\mathfrak T$ with vertices 
$(1/4, 1/4),$ $P_1,$ and $(1/2,1/2)$ (see Figure \ref{fig1}).  The first estimate in \eqref{al<j<k<2j-al} with $2^n\sim 1$ gives
\[
\Vert \mathfrak{M}_{c}f_j^k\Vert_{L^q}\lesssim 
2^{j(-\frac{1}{2}+\frac{1}{2p}+\frac{1}{2q})}2^{k(-\frac{1}{2}+\frac{3}{2p}-\frac{1}{2q}+\epsilon)}\Vert f\Vert_{L^p}, \quad  {1}/{p}+{3}/{q}\ge 1,
\]
for  $0\leq j\leq k\leq 2j$. $\mathfrak{M}_{c} f_{j}^{<j}$ satisfies the same bound with $k=j$. Note that 
$-{3}/{2}+{7}/(2p)-{1}/(2q)<0$,   $-1+2/p<0$, and ${1}/{p}+{3}/{q}> 1$ if  $(1/p,1/q)\in \operatorname{int}\mathfrak T$ (Figure \ref{fig1}). 
Thus, using those  estimates, we get 
\[ \tx \sum_{j\geq 0} \| \mathfrak S_j^1 f \|_{L^p}\lesssim  \sum_{j\geq 0} \big(  2^{j(-\frac{3}{2}+\frac{7}{2p}-\frac{1}{2q}+\epsilon)} + 2^{j(-1+\frac{2}{p}+\epsilon)} \big) \Vert f\Vert_{L^p} \lesssim \Vert f\Vert_{L^p}  \]
for $(1/p,1/q)\in \operatorname{int}\mathfrak T$.  
We now consider $\sum_{j\geq 0}  \mathfrak S_j^2 f$. By the first estimate in \eqref{al<j,2j-al<k} with $2^n\sim 1$, we have 
\[
 \tx \sum_{j\geq 0} \| \mathfrak S_j^2 f\|_{L^p}\lesssim  \sum_{0\le j, 2j<k}  2^{j(-\frac{1}{2}+\frac{3}{2p}-\frac{1}{2q}+\epsilon)}2^{k(-\frac{1}{2}+\frac{1}{p})}\Vert f\Vert_{L^p}
\lesssim \| f\|_{L^p} 
\]
for  $(1/p,1/q)\in \operatorname{int}\mathfrak T$.    Thus, $\mathfrak{M}_{c}  f_{\ge 0}$  is  bounded from  $L^p$ to $L^q$  for $(1/p, 1/q)\in \operatorname{int}\mathfrak T$.

Next, we show $L^p$--$L^q$ bound on  $\mathfrak{M}_{c}  f_{\ge 0}$ for $(1/p, 1/q)\in \operatorname{int} \mathcal Q'$ where $\mathcal Q'$ is the quadrangle with  vertices 
$(1/4, 1/4),$  $(0,0),$ $P_1,$ and $P_2$ (see Figure \ref{fig1}). 
Note that $1/p+3/q<1$ if $(p,q)\in  \operatorname{int} \mathcal Q'$.  By the second estimate of \eqref{al<j<k<2j-al} with $2^n\sim 1$, we have  
\[ 
 \Vert \mathfrak{M}_{c}f_j^k\Vert_{L^q}\lesssim  
2^{j(1-\frac{1}{p}-\frac{4}{q})}2^{k(-1+\frac{2}{p}+\frac{1}{q}+\epsilon)}\Vert f\Vert_{L^p}, \quad {1}/{p}+{3}/{q}<1
\]
for  $0\leq j\leq k\leq 2j$. $\mathfrak{M}_{c} f_{j}^{<j}$ satisfies the same bound with $k=j$.  Thus, 
\[ 
 \tx \sum_{j\geq 0}  \| \mathfrak S_j^1 f \|_{L^p}\lesssim  \sum_{j\geq 0}   ( 2^{j(\frac{1}{p}-\frac{3}{q}+\epsilon)} +2^{j(\frac{3}{p}-\frac{2}{q}-1+2\epsilon)})   \Vert f\Vert_{L^p} \lesssim  \Vert f\Vert_{L^p}
\]
for $(1/p,1/q)\in \operatorname{int} \mathcal Q'$ since ${1}/{p}-{3}/{q}<0$ and ${3}/{p}-{2}/{q} <1$ for 
$(1/p,1/q)\in \operatorname{int} \mathcal Q'$.  
Similarly, the second estimate of \eqref{al<j,2j-al<k} with $2^n\sim 1$ gives
\[
\tx
\sum_{j\geq 0}   \| \mathfrak S_j^2 f \|_{L^p}\lesssim \sum_{k> 2j\ge 0}2^{j(\frac{1}{p}-\frac{2}{q}+\epsilon)}2^{k(-\frac{1}{2}+\frac{1}{p})}\Vert f\Vert_{L^p} \lesssim \sum_{j\geq 0}  2^{j(-1+\frac{3}{p}-\frac{2}{q}+\epsilon)}\Vert f\Vert_{L^p}
\]
for $(1/p,1/q)\in \operatorname{int} \mathcal Q'$.  Note that $-1+{3}/{p}-{2}/{q}<0$ for $(1/p, 1/q)\in \operatorname{int} \mathcal Q'$, so  it follows that  $\sum_{j\ge 0} \| \mathfrak S_j^2 f \|_{L^p}\lesssim  \Vert f\Vert_{L^p}$ for   $(1/p,1/q)\in \operatorname{int} \mathcal Q'$.
Thus, $f\to \mathfrak{M}_{c}  f_{\ge 0}$  is bounded from $L^p$ to $L^q$ for $(1/p, 1/q)\in \operatorname{int} \mathcal Q'$. 

Consequently,   $f\to \mathfrak{M}_{c}  f_{\ge 0}$ is bounded from $L^p$ to $L^q$  for $(1/p, 1/q)\in \operatorname{int} \mathfrak T \cup  \operatorname{int} \mathcal Q'$. Thus,  via interpolation  $f\to \mathfrak{M}_{c}  f_{\ge 0}$ is bounded  from $L^p$ to $L^q$ for $(1/p, 1/q)\in \operatorname{int} \mathcal Q$.  This complete the proof of Theorem \ref{main thm}.

\subsection{Proof of Theorem \ref{two para smoothing}} 
We set $\mathbb D_\tau=\R^3\times \mathbb J_\tau$. 
By $L^{p}_{\alpha, {x}}$ we denote the $L^p$ Sobolev space of order $\alpha$ in $x$, and set
$\cL^{p}_{\alpha}(\mathbb D_{\tau}) =L^p_{s,t}( \mathbb J_\tau;  L^{p}_{\alpha, x} (\mathbb R^3))$. 
We  prove  Theorem \ref{two para smoothing} making use of the next lemma.

\begin{prop}
\label{prop:2psmoothing}
Let $\tau\in (0,1]$ and $8\le p<\infty$. If  $ \alpha <4/p$, then we have
\[
 \Vert \tilde{\mathcal A}_t^{s}f\Vert_{\sobn \tau}\lesssim \tau^{-\frac{3}{p}}\Vert f\Vert_{L^p}. 
  \]
\end{prop}

It is not difficult to see that  the bound $\tau^{-{3}/{p}}$ is sharp up to a constant  by using a frequency localized smooth function. 
Assuming Proposition \eqref{prop:2psmoothing} for the moment, we prove Theorem \ref{two para smoothing}.

\begin{proof}[Proof of Theorem \ref{two para smoothing}]  
Since  $\psi\in C_c^\infty (\mathbb J_\ast)$, as before,  there are constants $c_0\in (0,1)$, and   $m_1, m_2>0$  such that 
$\supp\psi \subset \{(t,s): m_1 \le  s\le m_2, s<c_0t\}$.  By finite decomposition and scaling, we may assume 
$\supp\psi\subset  \{(t,s): 1 \le  s\le 2, s<c_0t\}$.  

We now consider the Fourier transform of the function
$(x,t,s)\to   \tilde{\mathcal A}_t^{s}f(x)$:  
\[      F(\zeta)  =   S(\zeta)  \widehat f(\xi)  := \iiiint  e^{-i(t\tau+s\sigma+ \Phi_t^s(\theta, \phi)\cdot\xi)} \psi(t,s) \,  d\theta d\phi ds dt\,   \widehat f(\xi),          \] 
where $\zeta=(\xi, \tau, \sigma)$. 
Let us set $m^\alpha(\zeta)=(1+|\zeta |^2)^{\alpha/2}$,  $ \varphi_{\circ}= \varphi_{<0}(|\cdot|)$, and $\tilde \varphi_{\circ}= 1-\varphi_{\circ}$. 
To prove Theorem \ref{two para smoothing}, we need to show  $\| \mathcal F^{-1} (m^\alpha F)\|_{L^p}\lesssim   \Vert f\Vert_{L^p}.$ 
Since $\| \mathcal F^{-1} (\varphi_{\circ} m^\alpha  F)\|_{L^p}\lesssim   \Vert f\Vert_{L^p},$ we only have  to show 
\[\| \mathcal F^{-1} ( \tilde\varphi_\circ m^\alpha  F)\|_{L^p}\lesssim   \Vert f\Vert_{L^p}.\]

For a large positive constant $C$, we set $\varphi_\ast(\zeta) =  \varphi_{<0}({|\tau|}/ {C|\xi|})$ and $\varphi^\ast(\zeta) =  \varphi_{<0}({|\sigma|}/{C|\xi|}).$ 
We also set $\tilde \varphi_\ast=1-\varphi_\ast$ and $\tilde \varphi^\ast= 1-\varphi^\ast$. 
Thus, we  have 
\[    \varphi_\ast  \varphi^\ast +  \tilde \varphi_\ast \varphi^\ast   +   \varphi_\ast  \tilde \varphi^\ast   +
 \tilde\varphi_\ast \tilde \varphi^\ast =1.       
 \] 
If $|\tau|\ge C|\xi|$, integration by parts in $t$ gives  $|S(\zeta)| \lesssim (1+|\tau|)^{-N}$ for any $N$. 
Since $|\tau|\ge C|\xi|$ and $|\sigma|\le C|\xi|$ on the support of $\tilde \varphi_\ast \varphi^\ast$, one can easily see 
$\| \mathcal F^{-1} ( \tilde \varphi_\ast \varphi^\ast \tilde\varphi^\circ m^\alpha  F)\|_{L^p}\lesssim   \Vert f\Vert_{L^p}$  for any $\alpha$. 
The same argument also shows  that $\| \mathcal F^{-1} (   \varphi_\ast  \tilde \varphi^\ast  \tilde\varphi^\circ m^\alpha  F)\|_{L^p}$, 
$ \| \mathcal F^{-1} (    \tilde\varphi_\ast \tilde \varphi^\ast  \tilde\varphi^\circ m^\alpha  F)\|_{L^p} \lesssim   \Vert f\Vert_{L^p}$  for any $\alpha$. 
Now, we note that $|\tau|\le C|\xi|$ and $|\sigma|\le C|\xi|$ on the support of $\varphi_\ast \varphi^\ast$. Thus, by the Mikhlin multiplier theorem 
\[   \| \mathcal F^{-1} ( \varphi_\ast \varphi^\ast  \tilde\varphi^\circ m^\alpha  F)\|_{L^p} \lesssim   \| \mathcal F^{-1} ( \bar m^\alpha   F)\|_{L^p},\]
where $\bar m^\alpha(\zeta)=(1+|\xi |^2)^{\alpha/2}.$  Since $\supp\psi\subset  \{(t,s): 1 \le  s\le 2, s<c_0t\}$, the right hand side is bounded above by 
$\Vert \tilde{\mathcal A}_t^{s}f\Vert_{\sobn 1}$. Therefore, using Proposition \ref{prop:2psmoothing}, we get  $\| \mathcal F^{-1} ( \varphi_\ast \varphi^\ast  \tilde\varphi^\circ m^\alpha  F)\|_{L^p}
\lesssim  \Vert f\Vert_{L^p}.$ \end{proof}

In what follows, we prove Proposition \ref{prop:2psmoothing} using the estimates obtained  in Section \ref{sec2.2}.

\begin{proof}[Proof of Proposition \ref{prop:2psmoothing}] 
Let $n$ be an integer such that $2^n\le 1/\tau< 2^{n+1}$. Then, we decompose 
\begin{align}\label{A_ts decomposition}
 \mathcal A_t^{s}f = \mathcal A_t^{s} f_{{<n}}^{{<n}}+  \sum_{k\geq n}\mathcal A_t^{s} f_{{<0}}^k +  \sum_{0\le j<n\le k}\mathcal A_t^{s} f_j^k +  {\mathrm I_{t}^s f}+   {\mathrm I\!\mathrm I_{t}^s f}, 
\end{align}
where 
\[   {\mathrm I_{t}^s f}=\sum_{j\geq n,\, k>2j-n}\mathcal A_t^{s} f_j^k, \qquad  \quad    {\mathrm I\!\mathrm I_{t}^s f}=\sum_{n\leq j\leq k\leq 2j-n}\mathcal A_t^{s}f_j^k +  \sum_{n\leq j }\mathcal A_t^{s}f_j^{<j}.\] 
Note that  $\Vert \mathcal A_t^{s} f_{<n}^{<n}\Vert_{L_x^{p, \alpha}}\lesssim \tau^{-\alpha} \Vert \mathcal A_t^{s}f\Vert_{L^{p}_x} $. So, 
$\Vert \mathcal A_t^{s} f_{<n}^{<n}\Vert_{\mathcal L^{p, \alpha}(\R^3\times \mathbb J_\tau)}\lesssim \tau^{-\alpha+1/p} \Vert  f\Vert_{L^{p}} \lesssim \tau^{-3/p} \Vert  f\Vert_{L^{p}}  $ since 
$\alpha<4/p$. 
Similarly, using \eqref{000st}, we  have $  \Vert \mathcal A_t^{s} f_{{<0}}^k\Vert_{\mathcal L^{p, \alpha} (\R^3\times \mathbb J_\tau)}\lesssim \tau^{1/p-1/2} 2^{(\alpha-1/2) k} 
\Vert f\Vert_{L^p} $
for $ k\geq n$. Taking sum over $k$ gives
\[ \tx  \Vert  \sum_{k\geq n} \mathcal A_t^{s} f_{{<0}}^k\Vert_{\mathcal L^{p, \alpha} (\R^3\times \mathbb J_\tau)}\lesssim  \sum_{k\geq n}  2^{(\alpha-\frac12) k} 
\tau^{\frac1p-\frac12} \Vert f\Vert_{L^p}\lesssim  \tau^{-3/p} \Vert  f\Vert_{L^{p}} \]
since $\alpha<4/p$ and $p>8$.  When $0\le j<n\le k$, by \eqref{0jnk local} it follows that  $\|\mathcal A_t^{s} f_j^k\Vert_{\mathcal L^{p, \alpha} (\R^3\times \mathbb J_\tau)}\lesssim  
\tau^{\frac1p-\frac12}  2^{j(-\frac{2}{p}+\epsilon)+k(\alpha-\frac{1}{2})}\|f\|_{L^p}$ for $p\ge 4$. 
Thus, we see that 
\[ \tx \| \sum_{0\le j<n\le k}\mathcal A_t^{s} f_j^k \|_{\mathcal L^{p, \alpha}(\R^3\times \mathbb J_\tau)}\lesssim  \tau^{\frac1p-\alpha}  \Vert  f\Vert_{L^{p}}\lesssim   \tau^{-\frac3p}  \Vert  f\Vert_{L^{p}}.  \]

Therefore, it remains to show the estimates for the operators $ {\rm I}_t^s$ and $ {\rm I\!I}_t^s$. Using $(c)$ and $(a)$ in Proposition \ref{A1 estimate}, we obtain,  respectively, 
\begin{align*}
\Vert \mathcal A_t^{s} f _j^k\Vert_{\mathcal L^{p, \alpha}(\R^3\times \mathbb J_\tau)}&\lesssim   \tau^{\frac1p-\frac12}    2^{j(-\frac{2}{p}+\epsilon)} 2^{ k(\alpha-\frac12)} \Vert f\Vert_{L^p}, \quad  && j\geq n, \, k> 2j-n, 
\\
\Vert \mathcal A_t^{s} f_j^k\Vert_{\mathcal L^{p, \alpha}(\R^3\times \mathbb J_\tau)}&\lesssim  \tau^{-\frac1p} 2^{j(1-\frac{6}{p})+k(\alpha+ \frac{2}{p}-1+\epsilon)}\Vert f\Vert_{L^p}, \quad &&n\leq j\leq k\leq 2j-n
\end{align*}
for any $\epsilon>0$ and $p\ge 4$. Besides,   $(b)$ in Proposition \ref{A1 estimate} (\eqref{A1 ineq1} with $h=\lambda$) gives $\Vert \mathcal A_t^{s}f_j^{<j}\Vert_{\mathcal L^{p, \alpha}(\R^3\times \mathbb J_\tau)}\lesssim \tau^{-1/p} 2^{j(\alpha-4/{p})} \Vert  f\Vert_{L^{p}}$ for $p\ge 4$.  Therefore,  recalling $p>8$ and $\alpha<4/p$,   we get  
\begin{align*}
&\tx \Vert \mathrm I_{t}^s f \Vert_{\mathcal L^{p, \alpha}(\R^3\times \mathbb J_\tau)} \lesssim   \tau^{\frac1p-\frac12} \sum_{j\geq n,\, k>2j-n}  2^{j(-\frac{2}{p}+\epsilon)} 2^{ k(\alpha-\frac12)}\Vert f\Vert_{L^p} \lesssim \tau^{-\frac3p} \Vert f\Vert_{L^p}, 
\\
&\tx  \Vert \mathrm I\!\mathrm I_{t}^s f \Vert_{\mathcal L^{p, \alpha}(\R^3\times \mathbb J_\tau)} \lesssim    \tau^{-\frac1p} \sum_{n\leq j\leq k\leq 2j-n} 2^{j(1-\frac{6}{p})+k(\alpha+ \frac{2}{p}-1+\epsilon)}\Vert f\Vert_{L^p}
\lesssim \tau^{-\frac3p} \Vert f\Vert_{L^p}.
\end{align*}
This completes the proof. 
\end{proof}

\section{One-parameter local smoothing and estimate with  fixed $t,s$ } \label{sec4}
In this section we prove Theorem \ref{smoothing} and \ref{fixed}.  

\subsection{One-parameter propagator}
In order to prove Theorem \ref{smoothing},  
we make use of  local smoothing estimate for the operator $f\to \mathcal{U}f({x},t,c_0t)$.
 For the two-parameter propagator $\mathcal{U}$, we can handle the associated operators $e^{it|\bar D|}$ and $e^{is|D|}$ separately
  so that  the sharp smoothing  estimates are obtained by utilizing  the decoupling and local smoothing inequalities for the cone in $\R^{2+1}$. 
  However, for the sharp estimate for  $f\to \mathcal{U}f({x},t,c_0t)$ a similar approach does not work. Instead, we make use of the decoupling inequality for the conic surface  $(\xi, |\bxi|+c_0|\xi|)$ in 
  $\R^{3+1}$. 
   (See \cite{BD} and Theorem 2.1 of \cite{BHS}). 
 
\begin{prop}\label{one para propagator} 
Set $  \tilde{\mathcal{U}}_\pm f(x,t)= \mathcal{U} f(x,t,\pm c_0t)$. 
Let $1\le \lambda\le  h\le \lambda^2$. Then, if  $6\le p\le \infty$,   for any $\epsilon>0$  we have
\Be 
\label{ut} \Vert \tilde{\mathcal{U}}_\pm f\Vert_{L^p_{{x},t}(\R^3\times[1,2])}\lesssim \lambda^{\frac{3}{2}-\frac{5}{p}} h^{\frac{2}{p}-\frac{1}{2}+\epsilon}\Vert f\Vert_{L^p} 
\Ee
 whenever $\supp \widehat f \subset \bA_\lambda\times \bI_h$. Also, the same bound with $h=\lambda$  holds   for $4\leq p\leq \infty$ whenever  $\supp \widehat f \subset \bA_\lambda\times \bI_\lambda^\circ$. 
\end{prop}

\begin{proof}
When $p=\infty$, the estimate \eqref{ut} is already shown  in the previous section  (see \eqref{plocal1}).  Thus, we focus on the estimates \eqref{ut}  for $p=4, 6$, and the other estimates follow by 
interpolation.

 We first consider the case $\supp \widehat f \subset \bA_\lambda\times \bI_\lambda^\circ$, for which \eqref{ut} hold on a larger range   $4\leq p\leq \infty$.  To show \eqref{ut},  we make use of the decoupling inequality associated to the conic surfaces 
 \[  \Gamma_\pm =\{  (\xi, P_\pm (\xi)),\quad \xi \in \mathbb A_1\times \mathbb I_1^\circ\} \]  
 where $P_\pm(\xi):=|\bxi|\pm c_0|{\xi}|$.  In fact, we use the $\ell^p$ decoupling inequality for the  conic surfaces \cite{BD, BHS}.  To this end, we  first check that the Hessian matrix of $P_\pm$ is of rank $2$. Indeed, a computation shows that
\[ 
\operatorname{Hess} P_\pm(\xi)
= \frac{1}{|\bar\xi|^3} \begin{pmatrix}
 \xi_2^2 & - {\xi_1\xi_2} & 0\\
- {\xi_1\xi_2} &  {\xi_1^2} & 0\\
0 & 0 & 0
\end{pmatrix}\pm \frac{c_0}{|{\xi}|^3}
\begin{pmatrix}
 {\xi_2^2+\xi_3^2} & - {\xi_1\xi_2} & - {\xi_1\xi_3}\\
- {\xi_1\xi_2} &  {\xi_1^2+\xi_3^2}& - {\xi_2\xi_3}\\
- {\xi_1\xi_3} & - {\xi_2\xi_3} &  {\xi_1^2+\xi_2^2}
\end{pmatrix}. \]
Note that  $\operatorname{Hess} P_\pm(\xi) \xi=0$, so $\Gamma$ has a vanishing principal   curvature in the direction of $\xi$. 
By rotational symmetry in $\bxi$, to compute the eigenvalues of $\operatorname{Hess} P_\pm(\xi)$  it is sufficient to consider the case $\xi_1=0$ and $\xi_2=|\bxi|\neq 0$. Consequently, 
one can easily see that the matrix $\operatorname{Hess} P_\pm(\xi)$  has two nonzero   eigenvalues 
\[{|\bxi|^{-1}}\pm {c_0} |\xi|^{-1},  \quad  \pm {c_0}|\xi|^{-1}.\] 

Let us denote by $\mathfrak{V}^\lambda$ a collection of points which are maximally $\sim \lambda^{-1/2}$  separated in the set $\mathbb S^2\cap \{\xi: |\bxi|\ge 2^{-2} \xi_3\} $.
Let $\{ W_\mu\}_{\mu\in  \mathfrak{V}^\lambda}$ denote a partition of unity subordinated to   a collection of finitely overlapping spherical caps centered at $\mu$  of diameter $\sim \lambda^{-1/2}$ 
which cover $\mathbb S^2\cap \{\xi: |\bxi|\ge 2^{-2} \xi_3\} $ such  that  $|\partial^\alpha  W_\mu|\lesssim \lambda^{|\alpha|/2}$. Denote $\Omega_\mu(\xi)=W_\mu(\xi/|\xi|)$. 
Since $\supp \widehat f \subset \bA_\lambda\times \bI_\lambda^\circ$, we have $f=\sum_{\mu\in  \mathfrak{V}^\lambda} f_\mu$ where 
$f_\mu=\mathcal F^{-1}  ( \Omega_\mu \widehat f \, ). $ So, 
we can write 
\[ \tilde{\mathcal{U}}_\pm f (x,t) =\sum_{\mu\in \mathfrak{V}^\lambda} \tilde{\mathcal{U}}_\pm f_\mu(x,t)=\sum_{\mu\in \mathfrak{V}^\lambda}\int e^{i({x}\cdot {\xi}+tP_\pm (\xi))}  \widehat{f_\mu}({\xi})d{\xi}.
\]
Since $\Gamma_\pm $ are conic surfaces with two nonvanishing curvatures in $\mathbb R^4$, we have the following $l^p$-decoupling inequality: 
\begin{align}
\label{decouple-p}
\Vert  \tilde \chi(t)  \tilde{\mathcal{U}}_\pm f\Vert_{L^p_{x,t}} 
\lesssim \lambda^{1-\frac3p+\epsilon} \Big( \sum_{\mu\in \mathfrak{V}^\lambda}\Vert  \tilde \chi(t)  \tilde{\mathcal{U}}_\pm f_\mu\Vert_{L^p_{x,t}}^p\Big)^{1/{p}}
\end{align}
for $p\ge 4$. (See \cite{BD1} and \cite[Theorem 1.4]{BHS}.) Here $\tilde\chi \in \mathcal S(\mathbb R)$ such that $\tilde \chi\ge 1$ on $\bI$ and $\supp \mathcal F (\tilde\chi)\subset [-1/2,1/2]$.  Using  Lemma \ref{fff} as before, we see $\Vert   \tilde \chi(t)  \tilde{\mathcal{U}}_\pm  f_\mu \Vert_{L^p_{x,t}}\lesssim  \|  \tilde \chi(t)  e^{t (\bar D\cdot (\bar \mu/|\bar\mu|)\pm c_0D\cdot\mu)} f_\mu\Vert_{L^p_{x,t}}$ where $\mu=(\bar\mu, \mu_3)$. Thus, a change of variables gives
$ \Vert  \tilde \chi(t)  \tilde{\mathcal{U}}_\pm f_\mu\Vert_{L^p_{x,t}}\lesssim \Vert f_\mu\Vert_{L^p}$  for $1\le p\le \infty$. 
Since $(\sum_\mu \|f_\mu\|_p^p)\lesssim \|f\|_p$ for $p\ge 2$, combining the  estimates and \eqref{decouple-p} with $p=4$, 
 we obtain 
\[ \Vert \mathcal{U}_\pm f  \Vert_{L^4_{x,t}}  \lesssim \lambda^{\frac1{4}+\epsilon}\Vert f\Vert_{L^4}.\]
  Interpolation with the easy $L^\infty$ estimate (\eqref{plocal1} with $p=q=\infty$) gives \eqref{ut} with $h=\lambda$ for $4\le p\le \infty$.

Now, we consider the case $\supp \widehat f \subset \bA_\lambda\times \bI_h$ with $\lambda\le h\le \lambda^2$. Recall the partition of unity $\{ w_\nu\}_{\nu\in \mathfrak{V}_\lambda}$  on the unit circle  $\mathbb S^1$ and $f_\nu=\omega_\nu(\bar D)  f$.  Note that $\tilde{\mathcal U}_\pm f_\nu(\cdot, x_3, t)$, $\nu\in \mathfrak{V}_\lambda$  have Fourier supports contained in  finitely overlapping rectangles of dimension $\lambda\times \lambda^{1/2}$. So, we have 
\[\tx\| \sum_{\nu\in \mathfrak{V}_\lambda} \tilde{\mathcal U}_\pm f_\nu(\cdot, x_3, t)\|_p \lesssim  \lambda^{1/2-1/p} (\sum_{\nu\in \mathfrak{V}_\lambda} \| \tilde{\mathcal U}_\pm f_\nu(\cdot, x_3, t)\|_p^p)^{1/p}\] for $2\le p\le \infty,$ which is a simple consequence of the Plancherel theorem and interpolation (for example, see Lemma 6.1 in \cite{TVV}). Integration in $x_3$ and $t$ gives 
\begin{equation}\label{after tvv}
\Vert  \tilde{\mathcal U}_\pm  f \Vert_{L^p_{{x},t}(\mathbb R^3\times \bI)}\lesssim \lambda^{\frac{1}{2}-\frac{1}{p}}\Big(\sum_{\nu\in \mathfrak{V}_\lambda} \| \tilde{\mathcal U}_\pm f_\nu\Vert_{L^p_{{x},t}(\mathbb R^3\times \bI)}^p\Big)^{1/{p}}, \quad 2\le p\le \infty.
\end{equation}
We proceed to obtain estimates for $ \| \tilde{\mathcal U}_\pm f_\nu\Vert_{L^p_{{x},t}(\mathbb R^3\times \bI)}$. Using Lemma \ref{fff} and changing variables $x\to x-(\nu,0)t$, we see  
$ \| \tilde{\mathcal U}_\pm f_\nu\Vert_{L^p_{{x},t}(\mathbb R^3\times \bI)}\lesssim  \| e^{\pm itc_0 |D|} f_\nu\Vert_{L^p_{{x},t}(\mathbb R^3\times \bI)}$. 
Similarly, we also have 
$ \| e^{\pm itc_0 |D|} f_\nu\Vert_{L^p_{{x},t}(\mathbb R^3\times \bI)}\lesssim \| \tilde{\mathcal U}^\nu_\pm f_\nu\Vert_{L^p_{{x},t}(\mathbb R^3\times \bI)} $, where 
\[   \tilde{\mathcal U}_\pm^\nu h ({x},t)= \int e^{i\big({x}\cdot{\xi} 
\pm c_0t\sqrt{(\nu\cdot\bxi)^2+\xi_3^2})\big)}\widehat{h}({\xi})d{\xi}. \]
Therefore,  from \eqref{after tvv} it follows that 
\begin{equation}\label{after tvv2}
\Vert  \tilde{\mathcal U}_\pm f \Vert_{L^p_{{x},t}(\mathbb R^3\times \bI)}\lesssim \lambda^{\frac{1}{2}-\frac{1}{p}}\Big(\sum_{\nu\in \mathfrak{V}_\lambda} \| \tilde{\mathcal U}_\pm^\nu f_\nu\Vert_{L^p_{{x},t}(\mathbb R^3\times \bI)}^p\Big)^{1/{p}}, \quad 2\le p\le \infty. 
\end{equation}

Note that  Fourier transform of $f$ is contained in $\{\xi :  |\xi|\sim h \}$ because $\lambda\le h$. To estimate $ \tilde{\mathcal U}_\pm^\nu f_\nu $, freezing $\nu^\perp\cdot \bx$,  we use  the $\ell^2$ decoupling inequality \cite{BD} (i.e., \eqref{BD-} with $p=2$, $q=6$, and $\lambda=h$) with respect to $\nu\cdot\bx,x_3$ variables.  Thus,  by the decoupling inequality followed by  Minkowski's inequality, we get 
\[  \|  \tilde{\mathcal U}_\pm^\nu f_\nu\Vert_{L^6_{{x},t}(\mathbb R^3\times \bI)}\lesssim h^{\epsilon} \Big( \sum_{\tilde\nu\in\mathfrak{V}_h }\Vert   \tilde\chi(t)  \tilde{\mathcal U}_\pm^\nu f_{\nu}^{\tilde \nu} \Vert_{L^6_{x,t}}^2 \Big)^{1/2},
 \]
where $\mathcal F(f_{\nu}^{\tilde \nu})(\xi)=\omega_{\tilde \nu}(\nu\cdot\bxi, \xi_3)\widehat{f_\nu}(\xi)$.   Since  $\#\{\tilde\nu: f_{\nu}^{\tilde \nu}\neq 0\}\lesssim \lambda h^{-1/2}$, by H\"older's inequality  it follows that 
\[ \|  \tilde{\mathcal U}_\pm^\nu f_\nu\Vert_{L^6_{{x},t}(\mathbb R^3\times \bI)}\lesssim  h^{\epsilon}  (\lambda h^{-1/2})^{\frac13} \Big( \sum_{\tilde\nu\in\mathfrak{V}_h }\Vert \tilde\chi(t)  \tilde{\mathcal U}_\pm^\nu f_{\nu}^{\tilde \nu} \Vert_{L^6_{x,t}}^6 \Big)^{1/6}. 
 \]
Lemma \ref{fff} and a similar argument as before yield $\Vert \tilde\chi(t) \tilde{\mathcal U}  f_{\nu}^{\tilde \nu} \Vert_{L^6_{x,t}} \lesssim \|f_{\nu}^{\tilde \nu} \|_6$. 
 Hence, $\|  \tilde{\mathcal U}_\pm^\nu f_\nu\Vert_{L^p_{{x},t}(\mathbb R^3\times \bI)}^6 $ $\lesssim  \lambda^{2} h^{-1+6\epsilon} 
 \sum_{\tilde\nu\in\mathfrak{V}_h } \| f_{\nu}^{\tilde \nu} \Vert_{L^6_{x,t}}^6 \lesssim   \lambda^{2} h^{-1+6\epsilon} \|f_\nu\|_{L^6}^6$. Therefore, combining this and \eqref{after tvv2} with $p=6$, 
we obtain  \eqref{ut} for $p=6$.  
\end{proof}

\subsection{Proof of Theorem \ref{smoothing}}
We denote  $\mathcal L^{p}_{\alpha}(\mathbb R^3 \times \mathbb I) =L^p_{t}( \mathbb{I}; L^{p}_{\alpha, x} (\mathbb R^3))$. 
By an argument similar to the proof of Theorem \ref{two para smoothing} it is sufficient to show that 
\[  \| \tilde \cA_t^{c_0t} f\|_{\mathcal L^p_\alpha(\mathbb R^3\times \bI)} \lesssim \|f\|_{L^p(\mathbb R^3)}, \quad \alpha<3/p \]
for a constant $c_0\in (0,1)$. 
We use the decomposition \eqref{A_ts decomposition} with $s=c_0t$ and $n=0$ to have 
\[
\tx \mathcal A_t^{c_0t}f = \mathcal A_t^{c_0t} f_{{<0}}^{{<0}}+  \sum_{k\geq 0}\mathcal A_t^{c_0t} f_{{<0}}^k 
+ \mathrm I_{t}^{c_0t} f+   \mathrm I\!\mathrm I_{t}^{c_0t} f. 
\]

The estimates for $\mathcal A_t^{c_0t} f_{{<0}}^{{<0}}$  and $\sum_{k\geq 0} \mathcal A_t^{c_0t} f_{{<0}}^k$ follow from  \eqref{00} and \eqref{000} for fixed $t,s$. Indeed, we have $\|\mathcal A_t^{c_0t} f_{{<0}}^{{<0}}\Vert_{\mathcal L^{p, 3/p} (\R^3\times \mathbb I)}\lesssim \|f\|_p $ and 
\[ \tx \sum_{k\ge 0}  \| \mathcal A_t^{c_0t} f_{{<0}}^{{k}} \Vert_{\mathcal L^{p, 3/p} (\R^3\times \mathbb I)} \lesssim \sum_{k\ge 0}  2^{(3/p-1/2)k} \|f\|_p\lesssim  \|f\|_p  \] 
for $p>6$.  

We obtain the estimates for $ \mathrm I_{t}^{c_0t}$ and $\mathrm I\!\mathrm I_{t}^{c_0t}$ using the next proposition. 

\begin{prop}\label{Bt case2}
$(a)$ If $1\le \lambda\le h\le \lambda^2$,  then for any $\epsilon>0$  we have 
\begin{equation}
\label{Bb}
\Vert \cA_t^{c_0t} f \Vert_{L^p_{{x},t}(\mathbb R^3\times \bI)}\lesssim \lambda^{1-\frac{5}{p}} h^{-1+\frac{2}{p}+\epsilon}\Vert f\Vert_{L^p}
\end{equation}
for $6\le p\le \infty$ whenever $\supp \widehat f \subset\bA_\lambda \times \bI_h$.  $(b)$
If $\supp \widehat f \subset\bA_\lambda \times \bI_\lambda^\circ$, the estimate \eqref{Bb} holds with $h=\lambda$ for $4\le p\le \infty$.  $(c)$ If $1\le \lambda$ and $\lambda^2\le h$,   we have
\[
\Vert \cA_t^{c_0t} f\Vert_{L^p_{{x},t}(\mathbb R^3\times \bI)}\lesssim \lambda^{-\frac{2}{p}+\epsilon} h^{-\frac{1}{2}}\Vert f\Vert_{L^p}
\]
for $4\leq p\leq \infty$ whenever $\supp \widehat f \subset\bA_\lambda \times \bI_h$.
\end{prop}

Assuming this for the moment, we finish the proof of   Theorem \ref{smoothing}.  
By $(a)$ and $(b)$ in Proposition \ref{Bt case2}  we have  
\[\tx  \| \mathrm I\!\mathrm I_{t}^{c_0t} f \|_{\mathcal L^p_\alpha(\mathbb R^3\times \bI)} 
\lesssim  \sum_{j\ge 0} 2^{(1-\frac{5}{p})j} \sum_{ j\leq k\leq 2j}  2^{k (-1+ \frac2p +\alpha+\epsilon)} \Vert f\Vert_{L^p}.\]
Since $p>6$ and $\alpha<3/p$, taking $\epsilon>0$ small enough, we have the right hand side bounded above by $C\|f\Vert_{L^p}$. 
Finally, using $(c)$ in Proposition \ref{Bt case2}  we obtain 
\[ \tx  \| \mathrm I_{t}^{c_0t} f \|_{\mathcal L^p_\alpha(\mathbb R^3\times \bI)} \lesssim \sum_{j\geq 0}\sum_{k\geq 2j}2^{j(-\frac{2}{p}+\epsilon)+k(-\frac{1}{2}+\alpha)}\Vert f\Vert_{L^p} \lesssim \Vert f\Vert_{L^p}\]
for $p>6$ and $\alpha<3/p$.

To complete the proof, it remains to prove Proposition \ref{Bt case2}. For the purpose  
we closely follow the proof of Proposition \ref{A1 estimate}.

\begin{proof}[Proof of Proposition \ref{Bt case2}] 
We recall \eqref{exp}, \eqref{ml}, and \eqref{ce}. 
As seen in the proof of Proposition \ref{A1 estimate}, using the Mikhlin multiplier theorem, 
we can handle $\mathcal E (\xi,t,c_0t)$  as if it is  $|\bxi|^{-N'} |\xi_3|^{-1}$ (see \eqref{bessel2}). 
Likewise, we can replace $\tilde E_N(c_0t|{\xi}|)$ by $(c_0t|{\xi}|)^{-N'}$.  
Thus, the matter is reduced to obtaining estimates for  the operators 
\[\tilde{\mathcal{C}}_\pm^\kappa  f(x,t)  :=|\bar D|^{-\frac12}|sD|^{-\frac12} e^{i(\kappa t|\bar D|\pm c_0t| D| )} f(x), \quad \kappa = \pm\]
(cf. \eqref{central}). 
 Thus, it is sufficient to show that the desired bounds on $\cA_t^{c_0t}$  also hold on $\tilde{\mathcal{C}}_\pm^\kappa$. 

We first consider the case $(a)$. Note  $\|\tilde{\mathcal{C}}_\pm^\kappa f \Vert_{L^p_{{x}}(\mathbb R^3)}\lesssim  (\lambda h) ^{-1/2} \Vert e^{i(\kappa t|\bar D|\pm c_0t| D| )} f\Vert_{L^p_{{x}}(\mathbb R^3)}$  since $\supp \widehat f \subset\bA_\lambda \times \bI_h$. By Proposition \ref{one para propagator} 
we get  
\[ \|\tilde{\mathcal{C}}_\pm^\kappa f \Vert_{L^p_{{x},t}(\mathbb R^3\times \bI)}\lesssim  \lambda^{1-\frac{5}{p}} h^{\frac{2}{p}-1+\epsilon}\Vert f\Vert_{L^p}, \quad \kappa = \pm\]
for $6\le p\le \infty$ as desired. In fact, the estimates for $e^{i(- t|\bar D|\pm c_0t| D| )} f$ follow by conjugation and reflection as before (cf. Remark \ref{pm}). 
Also,  note that $\|\tilde{\mathcal{C}}_\pm^\kappa f \Vert_{L^p_{{x}}}\lesssim \lambda^{-2} \Vert e^{i(- t|\bar D|\pm c_0t| D| )}\Vert_{L^p_{{x}}}$  when $\supp\widehat f \subset\bA_\lambda \times \bI_\lambda^\circ$. Thus, we get  the estimate in the case $(b)$  in the same manner.

Finally, we consider the case  $(c)$. Since  $\supp \widehat f \subset\bA_\lambda \times \bI_h$ and  $\lambda^2\le h$, applying  Mikhlin's multiplier theorem and Lemma \ref{fff} successively,  we see $\|\tilde{\mathcal{C}}_\pm^\kappa f\|_{L_x^p}\lesssim (\lambda h) ^{-1/2} \|e^{i(\kappa t|\bar D|\pm c_0t| D| )} f\|_{L_x^p} \lesssim (\lambda h) ^{-1/2} \|e^{i(\kappa t |\bar D|\pm c_0 D_3 )} f\|_{L_x^p}$. Thus, by a change of variables we have 
\[  \| \tilde{\mathcal{C}}_\pm^\kappa f \Vert_{L^p_{{x},t}(\mathbb R^3\times \bI)}\lesssim  (\lambda h) ^{-1/2} \Vert  e^{i\kappa t|\bar D|} f \Vert_{L^p_{{x},t}(\mathbb R^3\times \bI)}\] 
for $1\le p\le \infty$ and $\kappa=\pm$. 
Therefore, for $4\leq p\leq \infty$, the desired estimate follows from \eqref{locals-est}. 
\end{proof}

\subsection{Estimates with fixed $s,t$}  
In this subsection we prove Theorem \ref{fixed}. 
We consider estimates for $\mathcal A_t^{s}$ with fixed $0< s<t$. 

\begin{lem} 
\label{case20-fix}
Let $0< s<t$. 
Let $1\le p\le q\le \infty$, $1/p+1/q\le 1$, and $h\ge \lambda \sim 1$.  Suppose $\supp \widehat f \subset \bA_\lambda \times \bI_h$. Then, we have
\[  \Vert \mathcal A_t^{s} f \Vert_{L^q_{{x}}}\lesssim  h^{\frac{1}{p}-\frac{1}{q}-\frac12}  \Vert f\Vert_{L^p}. \]
\end{lem}

\begin{proof}
 Recalling  \eqref{exp}, \eqref{ml}, and \eqref{ce}, we see that the main contribution comes from  ${\mathcal{C}}_\pm^\kappa $ (see \eqref{central}). 
Applying  Mikhlin's theorem and Lemma \ref{fff}, we see that  $\|{\mathcal{C}}_\pm^\kappa  f(\cdot, t,s)\|_{L^q_{{x}}}\lesssim  h^{-1/2} \| e^{\pm is| D|}  f  \|_{L^q_{{x}}}
\lesssim h^{-1/2}   \| e^{\pm is |D_3|}  f  \|_{L^q_{{x}}}$.
Thus, Bernstein's inequality gives the desired estimate  $\|\tilde{\mathcal{C}}_\pm^\kappa  f\|_{L^q_{{x}}}\lesssim h^{\frac{1}{p}-\frac{1}{q}-\frac12}  \Vert f\Vert_{L^p}$ since $\supp \widehat f \subset \bA_\lambda \times \bI_h$ and $\lambda\sim 1$.  
\end{proof}

\begin{lem}\label{A1 fixed estimate}
Let  $0< s<t$ and  $p\ge 2$.  $(a)$ If $1\le \lambda\le h\le \lambda^2$, then for any $\epsilon>0$
\Be\label{Bbb}  \Vert \mathcal A_t^{s} f\Vert_{L^p_{{x}}}\lesssim \lambda^{1-\frac{3}{p}} h^{-1+\frac{1}{p}+\epsilon}\Vert f\Vert_{L^p} \Ee
holds whenever $\supp \widehat f \subset\bA_\lambda \times \bI_h$. $(b)$ 
If $\supp \widehat f \subset\bA_\lambda \times \bI_\lambda^\circ$, we  have the estimate \eqref{Bbb} with $h=\lambda$.   $(c)$ If $1\le \lambda$ and $\lambda^2\le h$, then for any $\epsilon>0$
\[ \Vert \mathcal A_t^{s} f\Vert_{L^p_{{x}}}\lesssim  \lambda^{-\frac1p} h^{-\frac12+\epsilon}  \Vert  f\Vert_{L^p} \]
holds whenever $\supp \widehat f \subset\bA_\lambda \times \bI_h$.
\end{lem}

\begin{proof}  As before,  it is sufficient to show that 
 ${\mathcal{C}}_\pm^\kappa $ (\eqref{central})  satisfies the above estimates in place of $ \mathcal A_t^{s}$.  
Note that  
\[ 
\| {\mathcal{C}}_\pm^\kappa f\|_{L^q_{{x}}}\lesssim  (\lambda h)^{-1/2} \| \mathcal Uf(\cdot, \kappa t,\pm s)  \|_{L^q_{{x}}}.
\] 

For all the cases $(a)$, $(b)$, and $(c)$, the desired estimates for $p=2$ follows by Plancherel's theorem. Thus,  we only need to show the estimates for $p=\infty$.    For the cases $(a)$ and $(b)$  the estimates for $p=\infty$ follow from  $\eqref{plocal1}$ of the corresponding cases  $(a)$ and $(b)$ with $p=q=\infty$ (Remark \ref{pm}).    Since  $\supp \widehat f \subset\bA_\lambda \times \bI_h$ and  $1\le \lambda$ and $\lambda^2\le h$, by Lemma \ref{fff} we note that $\| \mathcal Uf(\cdot, \kappa t, \pm s)  f \|_{L^\infty_{{x}}}\lesssim   \| e^{i(\kappa t|\bar D|\pm s|D_3|)} f \|_{L^\infty_{{x}}}\lesssim  \sum_{\pm}  \| e^{it|\bar D|} f_\pm \|_{L^\infty_{{x}}}$ where $\widehat f_\pm(\xi) = \chi_{(0, \infty)}(\pm \xi_2) \widehat f(\xi)$. Since $\supp \widehat f \subset\bA_\lambda \times \bI_h$,  the estimate for $p=\infty$ in the case $(c)$ follows from \eqref{locals-est}. 
\end{proof}

\begin{proof}[Proof of Theorem \ref{fixed}]  
Since $\mathcal A_t^sf $ is bounded from $L^2$ to $L^2_{1/2}$, it is sufficient to show $\mathcal A_t^sf $ is bounded from $L^p$ to $L^p_{\alpha}$ for  $p>4$ and $\alpha>2/p$. 

We use the decomposition \eqref{A_ts decomposition} with $2^n\sim 1$.  Note that $\| \mathcal A_t^s f_{{<0}}^{{<0}}\|_{L^{p}_{\alpha,{x}}}\lesssim \| \mathcal A_t^s f_{{<0}}^{{<0}}\|_{L^{p}_x} $ and $\| \mathcal A_t^s f_{{<0}}^{k}\|_{L^{p}_{\alpha,{x}}}\lesssim  
2^{\alpha k} \| \mathcal A_t^s f_{{<0}}^{k} \|_{L^{p}_x}$. By Lemma \ref{case20-fix} we have 
\[ \tx  \| \mathcal A_t^s f_{{<0}}^{{<0}}\|_{L^{p}_{\alpha,{x}}}+\sum_{k\ge 0}   \| \mathcal A_t^s f_{{<0}}^{k}\|_{L^{p}_{\alpha,{x}}}\lesssim  
\sum_{k\ge 0} 2^{(\alpha-1/2)k} \|f\|_{L^p} \lesssim \|f\|_p\] 
  for  $\alpha<2/p$ and $p> 4$.  Since  $\alpha<2/p$, using $(a)$ and $(b)$ in Lemma \ref{A1 fixed estimate} with an $\epsilon$ small enough, 
  we have 
  \[  \tx   \|   \mathrm I\!\mathrm I_{t}^{s} f\|_{L^{p}_{\alpha,{x}}}\lesssim  \sum_{0\leq j\leq k\leq 2j} 2^{j(1-\frac{3}{p})} 2^{k(\alpha-1+\frac{1}{p}+\epsilon)} \|f\|_{L^p}  \lesssim \|f\|_{L^p} \]  
  for $p\ge 2$. Similarly, using $(c)$ in Lemma \ref{A1 fixed estimate},  we obtain 
\begin{align*}
\tx \|    \mathrm I_{t}^{s} f\|_{L^{p}_{\alpha,{x}}}  \lesssim  \sum_{j\geq 0}\sum_{k\ge 2j}  2^{(\alpha-\frac12+\epsilon)k} 2^{-\frac1p j}  \Vert f\Vert_{L^p}\lesssim  
\Vert f\Vert_{L^p} 
  \end{align*}
  for $p> 4$ and $\alpha<2/p$. 
\end{proof}

\section{Sharpness of the results}\label{counterexample section}
In this section,  considering specific examples,  we show sharpness of the estimates in  Theorem  \ref{main thm},  \ref{two para smoothing}, \ref{smoothing}, and \ref{fixed} 
except for some endpoint cases.

\subsection{Necessary conditions  on $(p,q)$ for \eqref{main est} to hold}
We show  that if \eqref{main est} holds, then the following hold true: 
\begin{equation*}
({\bf a}) \ \, p\leq q, \qquad ({\bf b}) \ \, 3+1/{q}\geq 7/{p}, \qquad ({\bf c}) \ \,  1+2/{q}\geq 3/{p},\qquad ({\bf  d}) \ \, 3/{q} \geq {1}/{p}.
\end{equation*}
This shows  that    \eqref{main est} fails unless $(1/p,1/q)$ is contained in the closure of ${\mathcal{Q}}$.

To show $(\mathbf a)$--$(\mathbf d)$, it is sufficient to consider $\mathfrak M_0$ (see \eqref{Mn}) instead of $\mathcal M_c$ with $\mathbb J_1=\{ (t,s)\in [1,2]^2:  s<c_0 t\}.$  
The condition $({\bf a})$ is clear since $\cA_t^s$ is an translation invariant operator, which can not be bounded from $L^p$ to $L^q$  if $p>q$. It can also be seen 
by a simple example.  Indeed, let $f_R$ be the characteristic function of a ball of radius $R\gg 1$ which is centered at the origin.
 Then, $ \mathfrak{M}_0 f_R(x)\sim 1$ for $|x|\le R/2$, so we have $ \Vert \mathfrak{M}_0f_R\Vert_{L^q}/\Vert f_R\Vert_{L^p}\gtrsim R^{3/q-3/p}$. Thus,   $ \mathfrak{M}_0$ can be bounded from $L^p$ to $L^q$ only if $p\leq q$. 

To show (\textbf{b}), let $f_r$ denote the characteristic function of the set
\[ \lbrace (x_1,x_2, x_3) : |x_1|<r^2, \ |x_2|<r, \ |x_3|<r^4 \rbrace \]
for a small  $r>0$. One can easily see  that $\mathfrak{M}_0f_r(x)\approx r^3$ if  $x_1\sim 1$, $|x_2|\lesssim r$, and $x_3\sim 1$.  This gives 
\[ {\Vert \mathfrak M_0 f_r\Vert_{L^q}}/{\Vert f_r\Vert_{L^p}}\gtrsim  r^{3+\frac{1}{q}-\frac{7}{p}}. \]
Therefore, letting $r\to 0$ shows that the maximal operator is bounded from $L^p$ to $L^q$ only if (\textbf{b}) holds.
Now, for (\textbf{c}) we consider  the characteristic function of 
\[ \lbrace (\bx, x_3) : ||\bx|-1|<r, \  |x_3|<r^2 \rbrace, \]
which we denote by  $\tilde f_r$. 
Note that $\mathfrak M_0\tilde f_r\sim r$ if $|\bx |\lesssim r$ and $x_3\sim 1$. So, we have 
\[ {\Vert \mathfrak M_0\tilde f_r\Vert_{L^q}}/{\Vert \tilde f_r\Vert_{L^p}}\gtrsim r^{1+\frac{2}{q}-\frac{3}{p}}, \]
which gives  (\textbf{c}) by taking $r\to 0$. Finally,  to show  (\textbf{d}), let $\bar f_r$  be the characteristic function of the $r$-neighborhood of $\mathbb T_{1}^{c_0}$. Then, $|\mathfrak M_0 \bar f_r(x)|\approx 1$ if $|x|\lesssim r$. Thus, it follows that $ {\Vert \mathfrak M_0 \bar f_r\Vert_{L^q}}/{\Vert \bar f_r\Vert_{L^p}}\gtrsim r^{\frac{3}{q}-\frac{1}{p}}. $
So,  letting $r\to 0$, we obtain  (\textbf{d}).

\subsection{Sharpness of smoothing estimates}
Let $c_0\in (0, 8/9)$, and let $\psi$ be a smooth function supported in $[1/2, 2]\times [(1-2^{-4})c_0, (1+2^{-3})c_0]$ such that 
$\psi=1$ if $(t,s)\in  [3/4, 7/4]\times [(1-2^{-5})c_0, (1+2^{-5})c_0]$.  Then, we consider 
\[ \tilde \cA_t^s f(x) =\psi(t,s) \cA_t^s f(x).\]

We first claim that the estimates \eqref{2p-sm}, \eqref{1p-sm}, and \eqref{eq:fixed} imply   $\alpha\le 4/p$, $\alpha\le 3/p$, and $\alpha\le 2/p$, respectively.  

Let $\zeta_0$ be a function  such that $\supp\widehat{\zeta_0}\subset [-10^{-2},10^{-2}]$ and $\zeta_0(s)>1$ if $|s|<c_1$ for a small constant $0<c_1\ll c_0$.   
Let $\zeta_\ast\in C_c([-2,2])$ such that $\zeta_\ast=1$ on $[-1,1]$.  Note that 
$\tilde {\mathbb{T}}_{1}^{c_0}:=\mathbb{T}_{1}^{c_0}\cap \{ x:| |\bx|-1|<10c_1, x_3>0 \}$ can be parametrized by a smooth radial function $\phi$. That is to say, 
\[\tilde {\mathbb{T}}_{1}^{c_0}=\{ (\bx, \phi(\bx)): ||\bx|-1|<10c_1\}.\]
For a large  $R\gg 1$, we consider 
\[f_R({x})=e^{iR(x_3+\phi(\bx))}\zeta_0\big(R(x_3+\phi(\bx))\big)\zeta_\ast(||\bx|-1|/c_1).\]

Then, we claim that 
\Be 
\label{claim} |\mathcal A_t^{s}f_R({x})|\gtrsim 1, \quad (x,t,s) \in S_R, 
\Ee
where 
$
S_R= \lbrace ({x},t,s) : |{x}|\le 1/(CR),   |t-1|\le 1/(CR),   |s-c_0|\le 1/(CR)\rbrace$ for a large constant $C>0$. 
Indeed, note that 
\begin{align*}
 \mathcal A_t^{s}f({x}) 
 =\int_{\mathbb{T}_t^{s}}e^{iR(x_3+ \phi(\bar y-\bx)-y_3)}\zeta_0(R(x_3+\phi(\bar y-\bx)-y_3))\zeta_\ast(||\bx-\bar y|-1|/c_1)d\sigma_t^{s}({y}).
\end{align*}
 If $|x|\le 1/(CR)$ and $||\bar y|-1|\le 2c_1$,   we  have  $|\phi(\bar y-\bx)-y_3|\lesssim 1/(CR)$ and $|x_3+\phi(\bar y-\bx)-y_3|\lesssim 1/(CR)$  when  $y_3=\phi(y)$, i.e., 
 $y\in  \tilde {\mathbb{T}}_{1}^{c_0}$. Thus, $ |\mathcal A_1^{c_0}f({x})|\sim 1$ if $ |{x}|\le 1/(CR)$. 
 Furthermore,  if $ |t-1|\le 1/(CR)$ and $|s-c_0|\le 1/(CR)$,  the integration is actually taken over a surface which is $O(1/(CR))$ perturbation of the surface $ \tilde {\mathbb{T}}_{1}^{c_0}$. 
 Thus, taking $C$ large enough,  we see that  \eqref{claim} holds.

 By Mikhlin's theorem it follows that $   \| \tilde \cA_t^{s} g\|_{L^p_\alpha(\mathbb R^5)} \gtrsim \Vert  (1+|D_3|^2)^{\alpha/2} \tilde{\mathcal A}_t^{s} g \|_{L^p_\alpha(\mathbb R^5)}$.   
Note that $\widehat{f_R}({\xi})=0$ if  $\xi_3\not\in [(1-10^{-2})R, (1+10^{-2})R]$.  Since $\mathcal F(\mathcal A_t^{s}f)({\xi})=\widehat{f}({\xi})  \mathcal F(d\sigma_t^{s})({\xi})$,  we see  
\[  \| \tilde \cA_t^{s} f_R\|_{L^p_\alpha(\mathbb R^5)}\gtrsim R^\alpha \Vert \mathcal A_t^{s}f_R\Vert_{L^p(\mathbb R^5)} \gtrsim R^\alpha \Vert \mathcal A_t^{s}f_R\Vert_{L^p(S_R)}\gtrsim R^{\alpha-5/p}. \]
For the last inequality we use \eqref{claim}. Since $\|f_R\|_{L^p}\sim R^{-1/p}$, \eqref{2p-sm} implies that $\alpha\leq 4/p$.  Fixing $t=1$ and $s=c_0$, by \eqref{claim} we similarly have $\Vert \cA_{1}^{c_0}f_R\Vert_{L^{p}_{\alpha,{x}}}\gtrsim R^{\alpha-3/p}$. 
Thus, \eqref{eq:fixed}  holds only if $\alpha\le 2/p$. Concerning $\cA_t^{c_0t}$, by \eqref{claim} it follows that $|\cA_t^{c_0t} f_R(x)|\gtrsim 1$ if $|t-1|\le /CR$ and $|x|\le 1/CR$ for $C$ large enough. 
Thus,  $\Vert \cA_t^{c_0t}  f_R\Vert_{L^{p,\alpha}_{{x},t}}\gtrsim R^{\alpha}\Vert \cA_t^{c_0t} f_R\Vert_{L^p_{{x},t}}\gtrsim R^{\alpha-4/{p}}.$
Therefore, \eqref{1p-sm} implies $\alpha\le 3/p$. This proves the claim.

Therefore, to show sharpness of the estimates \eqref{2p-sm}--\eqref{eq:fixed}. We need only to  show  that each of the estimates \eqref{2p-sm}, \eqref{1p-sm}, and \eqref{eq:fixed} holds only if $\alpha\le 1/2$.  To do this, we consider 
\[ g_R({x})= e^{iR(x_3+c_0)}\zeta_0(R(x_3+c_0))\zeta(|x|).\]
Then, we have 
\Be\label{claim2}
|\mathcal A_t^{s}g_R({x})|\gtrsim R^{-\frac{1}{2}}
\Ee
if $({x},t,s)\in \tilde S_R\coloneqq \lbrace ({x},t,s) : |x|, |t-1|,  |s-c_0| \le 1/C,  |x_3+c_0-s|\le 1/{CR} \rbrace$ for a large constant $C\gg c_0$.
Indeed, note  that 
\[ 
 \mathcal A_t^{s}g_R({x}) 
=\int_{\mathbb{T}_t^{s}}e^{iR(x_3+c_0-y_3)}\zeta_0(CR(x_3+c_0-y_3))\zeta(|x-y|)d\sigma_t^{s}(\bar{y}).
\]
Recalling \eqref{para}, we see  that the integral is nonzero only if $|R(x_3+c_0-s\sin \theta)|\le 2/CR$. 
Since $ |x_3+c_0-s|\le 1/{CR}$, the integral is taken over the set   $\tilde {\mathbb T}:=\{ \Phi_t^s(\theta, \phi): |1-\sin\theta|\lesssim 1/R\}$. Note that 
the surface area of $\tilde {\mathbb T}$ is about $R^{-1/2}$, thus \eqref{claim2} follows. 
Since  $\widehat{g_R}({\xi})=0$ if  $\xi_3\not\in [(1-10^{-2})R, (1+10^{-2})R]$,  following the same argument as above, from \eqref{claim2}  we obtain 
$\Vert \mathcal A_t^{s}g_R \Vert_{L^{p,\alpha}_{{x},t,s}} 
\gtrsim R^{\alpha}R^{-{1}/{2}-{1}/{p}}.$
Hence, \eqref{2p-sm}  implies that $\alpha\leq 1/{2}$. 

 Regarding  \eqref{1p-sm}, we consider $\tilde S_{\! R}^{'}\coloneqq \lbrace ({x},t,s) : |x|, |t-1| \le 1/C,  |x_3+c_0-c_0t|\le 1/{CR} \rbrace$ for a large constant $C\gg c_0$. 
 Then, we have $|\mathcal A_t^{c_0t}g_R({x})|\gtrsim R^{-{1}/{2}}$ for $(x,t)\in \tilde S_{\! R}^{'}$, thus we see \eqref{1p-sm} implies $\alpha\le 1/2$. 
 
Finally, for \eqref{eq:fixed}, fixing $t=1$ and $s=c_0$, we consider  $\bar S_R:=\lbrace {x} : |x|\le 1/C, |x_3|\le 1/{CR} \rbrace$ for a  constant $C>0$. Then,  it is easy to see  
 $|A_{1}^{c_0}g_R({x})|\gtrsim R^{-1/{2}}$ for ${x}\in \bar S_R$ if we take $C$ large enough. Similarly as before, we have 
$ \Vert A_{1}^{c_0}g_R \Vert_{L^{p}_{\alpha,{x}}}\gtrsim R^{\alpha}R^{-1/{2}-1/{p}}$.  Therefore, \eqref{eq:fixed} implies $\alpha\le 1/2 $  because $\|g_R\|_{L^p}\sim R^{-1/p}$.

\section*{ Acknowledgements}  
This work was  supported by the  National Research Foundation (Republic of Korea) grant  no. 2022R1A4A1018904.

\bibliographystyle{plain}

\begin{thebibliography}{C}

\bibitem{ajs} T. Anderson, K. Hughes, J.  Roos, A. Seeger, \emph{$L^p$--$L^q$  bounds for spherical maximal operators}, 
Math. Z. {\bf 297} (2021),  1057–1074.

\bibitem{BGHS}  D. Beltran, S. Guo, J. Hickman, A. Seeger, \emph{Sharp $L^p$ bounds for the helical maximal function,} 
arXiv:2102.0827. 

\bibitem{BHS} D. Beltran, J. Hickamn, C. D. Sogge, \textit{Variable coefficient Wolff-type inequalities and sharp local smoothing estimates for wave equations on manifolds}, 
Anal. PDE {\bf 13} (2020),  403--433.

\bibitem{B2} J. Bourgain, 
\textit{Averages in the plane over convex curves and maximal operators},
J. Anal. Math. {\bf 47} (1986), 69--85.

\bibitem{BD} J. Bourgain, C. Demeter, \textit{The proof of the $l^2$ decoupling conjecture}, Ann. of Math. (2) {\bf 182} (2015), 351--389.

\bibitem{BD1} \bysame, \emph{Decouplings for curves and hypersurfaces with nonzero Gaussian curvature,}  Journal d'Analyse Mathématique {\bf 133} (2017),  79–311.

\bibitem{BDIM} S. Buschenhenke, S. Dendrinos, I. A. Ikromov, D. M\"uller, \textit{Estimates for maximal functions associated to hypersurfaces in $\R^3$ with height $h<2$: Part I}, Trans. Amer. Math. Soc. {\bf 372} (2019),  1363--1406.


\bibitem{C} Y. Cho, \textit{Multiparameter maximal operators and square functions on product spaces}, Indiana Univ. Math. J. {\bf 43} (1994), 459--491.

\bibitem{E} B. M. Erdo\v gan, \textit{Mapping properties of the elliptic maximal function}, Rev. Mat. Iberoamericana {\bf 19} (2003), 221--234.

\bibitem{gr} M.  Greenblatt, \emph{$L^p$ boundedness of maximal averages over hypersurfaces in $\mathbb R^3$}, Trans. Amer. Math. Soc. {\bf 365} (2013), 1875–1900. 

\bibitem{GWZ} L. Guth, H. Wang, R. Zhang, \textit{A sharp square function estimate for the cone in $\mathbb R^3$},  Ann. of Math.  \textbf{192}  (2020), 551--581.

\bibitem{HKLO} S. Ham, H. Ko, S. Lee, S. Oh, \emph{Remarks on dimension of unions of curves}, arXiv:2208.03272. 

\bibitem{Hansen} W. Hansen, \emph{Littlewood's one-circle problem, revisited,} 
Expo. Math. \textbf{26} (2008), 365--374.


\bibitem{H} Y. Heo, \textit{Multi-parameter maximal operators associated with finite measures and arbitrary sets of parameters}, 
Integral Equations Operator Theory {\bf 86} (2016), 185–-208.


\bibitem{hormander-book} L. H\"ormander, \emph{The analysis of linear partial differential operators I: Distribution Theory
and Fourier Analysis}, 2nd ed., Springer-Verlag, 1990.

\bibitem{is} A. Iosevich, E. Sawyer, \emph{Maximal averages over surfaces}, Adv. Math. 132 (1997), 46–119. 

\bibitem{IKM} I. A. Ikromov, M. Kempe, D. M\"uller, 
\textit{Estimates for maximal functions associated with hypersurfaces in $\R^3$ and related problems of harmonic analysis}, Acta Math. {\bf 204} (2010), 151--271.

\bibitem{KLO} H. Ko, S. Lee, S. Oh, \textit{Maximal estimates for averages over space curves}, Invent. Math. {\bf 228} (2022), 991--1035.

\bibitem{KLO1} \bysame,   \emph{Sharp smoothing properties of averages over curves,} 	arXiv:2105.01628. 



\bibitem{L} S. Lee, \textit{Endpoint estimates for the circular maximal function}, Proc. Amer. Math. Soc. {\bf 131} (2003), 1433--1442.

\bibitem{MR} G. Maletta, F. Ricci, \textit{Two-parameter maximal functions associated with homogeneous surfaces in $\R^n$}, Studia Math. {\bf 130} (1998), 53--65.

\bibitem{MRZ} G. Maletta, F. Ricci, J. Zienkiewicz, \textit{Two-parameter maximal functions associated with degenerate homogeneous surfaces in $\R^3$}, Studia Math. {\bf 130} (1998),  67--75.

\bibitem{MSS2} G. Mockenhaupt, A. Seeger, C. Sogge, \textit{Wave front sets, local smoothing and Bourgain's circular maximal theorem},  Ann. of Math. \textbf{136} (1992), 207--218.

\bibitem{nsw} A. Nagel, A. Seeger, S.  Wainger, \emph{Averages over convex hypersurfaces}, Amer. J. Math. {\bf 115} (1993), 903–927. 


\bibitem{PS} M. Pramanik, A. Seeger, \textit{$L^p$ regularity of averages over curves and bounds for associated maximal operators}, Amer. J. Math. {\bf 129} (2007), 61--103.


\bibitem{Schlag97} W. Schlag,  \textit{A generalization of Bourgain's circular maximal theorem}, J. Amer. Math. Soc. {\bf 10} (1997), 103--122.
	
\bibitem{SS} W. Schlag, C. D. Sogge, \textit{Local smoothing estimates related to the circular maximal theorem}, Math. Res. Let. {\bf 4} (1997), 1--15.

\bibitem{RS} F. Ricci,  E. M. Stein, \emph{Multiparameter singular integrals and maximal functions}, Ann. Inst. Fourier (Grenoble) {\bf 42} (1992), 637–670.

\bibitem{sww} A. Seeger, S. Wainger, J. Wright, \emph{Spherical maximal operators on radial functions},  Math. Nachr. {\bf 187}, 95–105 (1997). 

\bibitem{Stein} E. M. Stein, \textit{Harmonic Analysis: Real Variable Methods, Orthogonality and Oscillatory Integrals}, Princeton Univ. Press, Princeton, NJ, 1993.

\bibitem{Stein2} \bysame, \textit{Maximal functions: spherical means},  Proc. Nat. Acad. Sci. USA {\bf 73} (1976), 2174--2175.

\bibitem{Talagrand}  M. Talagrand, \emph{Sur la mesure de la projection d’un compact et certaines familles de cercles}, Bull. Sci. Math.
{\bf 104} (1980), 225–231.


\bibitem{TVV} T. Tao, A. Vargas, L. Vega, \emph{A bilinear approach to the restriction and Kakeya conjectures},
J. Amer. Math. Soc. {\bf 11} (1998), 967–1000.

\bibitem{Wolff} T. Wolff, \emph{Lectures on harmonic analysis}, University Lecture Series, 29. American Mathematical Society, Providence, RI, 2003. 


\end{thebibliography}

\end{document}